\documentclass{amsart}
\usepackage{amsthm}
\usepackage{amssymb}
\usepackage{mathtools}
\usepackage{graphicx}
\usepackage{enumitem}
\usepackage[all]{xy}
\usepackage{tikz}
\usetikzlibrary{arrows}

\definecolor{lightgrey}{rgb}{.804,.804,.756}
\usepackage[pdfauthor={V. Lebed, L. Vendramin},
pdftitle={On structure groups of YBE solutions},%
colorlinks=false,linkbordercolor=lightgrey,citebordercolor=lightgrey,urlbordercolor=lightgrey]{hyperref}

\definecolor{myred}{rgb}{.545,0,0}
\definecolor{myblue}{rgb}{.024,.15,.645}    
\definecolor{mygreen}{rgb}{0,.455,0}

\newcommand{\N}{\mathbb{N}}
\newcommand{\Z}{\mathbb{Z}}

\renewcommand{\O}{\mathcal{O}}
\newcommand{\RO}{\mathcal{R}\mathcal{O}}
\newcommand{\cR}{\mathcal{R}}

\newcommand{\Ret}{\operatorname{Ret}}
\newcommand{\Orb}{\operatorname{Orb}}

\newcommand{\GL}{\mathbf{GL}}

\newcommand{\Id}{\operatorname{Id}}
\newcommand{\rk}{\operatorname{rk}}

\newcommand{\Sol}{\operatorname{Sol}}
\newcommand{\BQ}{\operatorname{BQ}}
\newcommand{\Rack}{\mathbf{Rack}}
\newcommand{\Quandle}{\mathbf{Quandle}}
\newcommand{\Grp}{\mathbf{Grp}}
\newcommand{\Q}{\operatorname{Q}}
\newcommand{\R}{\operatorname{R}}
\newcommand{\SolCat}{\mathbf{YBESol}}
\newcommand{\BQCat}{\mathbf{Biquandle}}
\newcommand{\InjSol}{\mathbf{InjSol}}
\newcommand{\InjRack}{\mathbf{InjRack}}
\newcommand{\Ab}{\operatorname{Ab}}
\newcommand{\IIS}{\operatorname{IIS}}
\newcommand{\Iso}{\operatorname{Iso}}
\newcommand{\Sym}{\operatorname{Sym}}
\newcommand{\oG}{\overline{G}}
\newcommand{\oJ}{\overline{J}}

\newcommand\op{\mathrel{\triangleleft}}
\newcommand\lop{\mathrel{\triangleright}}
\newcommand{\Sq}{\operatorname{Sq}}
\newcommand{\hsigma}{\widehat{\sigma}}
\newcommand{\htau}{\widehat{\tau}}
\newcommand\ract{\mathrel{\leftharpoonup}}
\newcommand\lact{\mathrel{\rightharpoonup}}
\newcommand\racts{\mathrel{\leftharpoondown}}
\newcommand\lacts{\mathrel{\rightharpoondown}}
\newcommand\ractsi{\mathrel{\hat{\leftharpoondown}}}
\newcommand\lactsi{\mathrel{\hat{\rightharpoondown}}}

\makeatletter
\numberwithin{equation}{section}
\numberwithin{figure}{section}
	\theoremstyle{plain}
\newtheorem{thm}{Theorem}[section]
\newtheorem{lem}[thm]{Lemma}
\newtheorem{cor}[thm]{Corollary}
\newtheorem{pro}[thm]{Proposition}
\newtheorem*{conjecture*}{Conjecture}

\newtheorem{question}[thm]{Question}
	\theoremstyle{definition}
\newtheorem{defn}[thm]{Definition}
\newtheorem{notation}[thm]{Notation}

\newtheorem{exa}[thm]{Example}
	\theoremstyle{remark}
\newtheorem{rem}[thm]{Remark}

\setlist{nolistsep}
\setenumerate{topsep=0pt}
\makeatother

\begin{document}

\title[On structure groups of YBE solutions]{On structure groups of set-theoretic solutions to the Yang--Baxter equation}

\begin{abstract} 
This paper explores the structure groups $G_{(X,r)}$ of finite non-degenerate set-theoretic solutions $(X,r)$ to the Yang--Baxter equation. Namely, we construct a finite quotient $\overline{G}_{(X,r)}$ of $G_{(X,r)}$, generalizing the Coxeter-like groups introduced by Dehornoy for involutive solutions. This yields a finitary setting for testing injectivity: if $X$ injects into $G_{(X,r)}$, then it also injects into $\overline{G}_{(X,r)}$. We shrink every solution to an injective one with the same structure group, and compute the rank of the abelianization of $G_{(X,r)}$. We show that multipermutation solutions are the only involutive solutions with diffuse structure group; that only free abelian structure groups are biorderable; and that for the structure group of a self-distributive solution, the following conditions are equivalent: biorderable, left-orderable, abelian, free abelian, torsion free.
\end{abstract}

\keywords{Yang--Baxter equation, structure group, structure rack, bijective $1$-cocycle, birack, quandle, biquandle, injective solution, multipermutation solution, diffuse group, orderable group, abelianization.}

\subjclass[2010]{16T25, 20N02, 06F15.}

\author{Victoria Lebed}
\address{
Hamilton Mathematics Institute and
School of Mathematics,
Trinity College,
Dublin 2, Ireland}
\email{lebed.victoria@gmail.com, lebed@maths.tcd.ie}
  
\author{Leandro Vendramin}
\address{
Departamento de Matem\'atica -- FCEN,
Universidad de Buenos Aires, Pab. I -- Ciudad Universitaria (1428),
Buenos Aires -- Argentina}
\email{lvendramin@dm.uba.ar}

\maketitle

\section*{Introduction}

The physics-motivated \emph{Yang--Baxter equation} is now omnipresent in mathematics. The interest to its set-theoretic version goes back to Drinfel$'$d \cite{MR1183474}. Compared to linear solutions, set-theoretic ones are easier to study and classify. At the same time, they form a rich family of structures, and their deformations yield a wide variety of linear solutions. Furthermore, one gets powerful knot and link invariants by counting diagram colorings by such solutions. In this paper, by a \emph{solution} we  mean a YBE solution
\[r(x,y)=(\sigma_x(y),\tau_y(x))\]
on a finite set $X$, with $r$ required to be invertible and \emph{non-degenerate} (i.e., the maps $\tau_y$ and $\sigma_y$ are invertible for all $y$).

Two types of solutions are particularly well studied: \emph{involutive} solutions (with $r^2=\Id$), and \emph{self-distributive (SD) solutions}, i.e., those of the form \[r_{\op}(x,y)=(y, x \op y),\] 
where $(X, \op)$ is a \emph{rack}, i.e., the binary operation~$\op$ is self-distributive and has invertible right translations. These two types can be thought of as perpendicular axes in the variety of solutions. Thus, to any solution one can associate its \emph{structure rack} $(X, \op_r)$ \cite{Sol,LYZ,LebVen}, which is trivial if and only if the solution is involutive. The interplay between these two axes is fundamental in our work. 

Following Etingof, Schedler, and Soloviev \cite{ESS}, consider the \emph{structure group}
\[G_{(X,r)} = \langle X \,|\, xy = \sigma_x(y)\tau_y(x) \text{ for all } x,y \in X \rangle\]
of a solution $(X,r)$. The group algebra of $G_{(X,r)}$ can be regarded  as the {universal enveloping algebra} of $(X,r)$. 

Structure groups bring group-theoretic tools into the study of the YBE. Thus, classifying structure groups (or certain quotients thereof) is a reasonable first step in the classification of solutions. This strategy was successfully implemented, for instance, by Ced{\'o}, Jespers, and del R{\'{\i}}o \cite{CJR_IYBG}. On the other hand, for involutive $r$, $G_{(X,r)}$ is a quadratic algebra with interesting properties, both geometric and algebraic: it is Bieberbach and of $I$-type (as shown by Gateva-Ivanova and Van den Bergh \cite{GIVdB}), and also Garside (according to Chouraqui \cite{MR2764830}). 

Extending the work of Chouraqui--Godelle \cite{ChouGod}, Dehornoy \cite{DehCycleSet} constructed a \emph{finite quotient} $\oG_{(X,r)}$ of $G_{(X,r)}$ by a normal free abelian subgroup; here $r$ is again involutive. It is analogous to the Coxeter group quotients for Artin--Tits groups, in that the solution $r$ and the Garside structure of $G_{(X,r)}$ reconstruct from $\oG_{(X,r)}$. 

Recently, the \emph{orderability} problem was solved for the groups $G=G_{(X,r)}$ with involutive $r$, first partially by Chouraqui \cite{ChouOrd}, and then completely by Bachiller, Ced{\'o}, and the second author \cite{BCV}. Their results read as follows:
\begin{itemize}
\item $G$ is bi-orderable $\; \Longleftrightarrow \; G$ is free abelian $\; \Longleftrightarrow \;r$ is trivial: $r(x,y)=(y,x)$;
\item $G$ is left-orderable $\; \Longleftrightarrow \; G$ is poly-$\Z$ $\; \Longleftrightarrow \; r$ is MP.
\end{itemize}
An involutive solution is called \emph{multipermutation (MP)} if several iterations of the retraction construction $X \mapsto \Ret(X)=X /\!\sim$ yield a one-element set; here $x \sim y$ means $\sigma_x = \sigma_y$, and $r$ induces a solution on $\Ret(X)$ \cite{ESS}. This gives a new example of a structural property of the solution (being MP) which can be read off its structure group (namely, its orderability). The importance of MP solutions is discussed, for instance, in \cite{MR2652212,MR2885602,V,BCJO}.
 
These results work for involutive $r$ only. General structure groups are much more mysterious. Among the above properties, only the existence of finite quotients was established for all solutions. Concretely, Lu, Yan, and Zhu \cite{LYZ} described a finitely generated abelian normal subgroup $Z^0_{(X,r)}$ of $G_{(X,r)}$ of finite index. Soloviev \cite{Sol} showed the rank of $Z^0_{(X,r)}$ to be $K_r=\#\Orb(X,\op_r)$, which is the number of \emph{orbits} of $X$ with respect to the actions $x \mapsto x \op_r y$. The  quotient $G^0=G^0_{(X,r)} = G_{(X,r)} / Z^0_{(X,r)}$ looses a lot of information about $(X,r)$. Thus, there are infinitely many solutions sharing the same $G^0$ \cite{CJR_IYBG}. In the involutive case, $G^0$ is only a quotient of Dehornoy's $\oG$, and no longer encodes $(X,r)$ faithfully. In the self-distributive case, a quotient $\oG$ of $G=G_{(X,{\op})}$ refining $G^0$ was described for $1$-orbit racks by Gra{\~n}a, Heckenberger, and the second author \cite{GraHeckVen}. Their construction is as close to faithfulness as possible: the natural map $X \to G$ is injective if and only if it remains so after passing to the quotient $X \to \oG$. 

This paper extends some of the above results to all solutions $(X,r)$. Theorem~\ref{thm:quotient} describes a normal free abelian subgroup $Z_{(X,r)}$ of $G_{(X,r)}$ of rank $K_r$, such that
\begin{itemize}
\item the quotient $\oG_{(X,r)} = G_{(X,r)} / Z_{(X,r)}$ is finite;
\item the natural map $X \to G_{(X,r)}$ is injective if and only if it remains so after passing to the quotient $X \to \oG_{(X,r)}$.
\end{itemize}
A solution is called \emph{injective} if the map $X \to G_{(X,r)}$ is so. Involutive solutions are archetypal examples: many of their properties, such as the group-theoretical characterization, generalize to all injective solutions \cite{Sol}. Our construction yields an injectivity test involving only finite objects $X$ and $\oG_{(X,r)}$, as opposed to the infinite group $G_{(X,r)}$. This answers a question of Soloviev \cite{Sol}. The subgroup $Z_{(X,r)}$ is generated by appropriately chosen powers $x^d \in G_{(X,\op_r)}$ of $x \in X$ (which coincide for $x$'s from the same $\op_r$-orbit), pulled back by the bijective $1$-cocycle $J \colon G_{(X,r)} \to G_{(X,\op_r)}$ from \cite{Sol,LYZ,LebVen}. For SD solutions, $\oG_{(X,r)}$ recovers the above quotient; for involutive solutions, it is a slight variation thereof. Our construction is explicit and the proofs are elementary; this is to be compared with the study of~$G^0$ in~\cite{Sol}, which required the Hochschild--Serre sequence.

As a by-product, we characterize finite injective racks as sub-racks of finite \emph{conjugation racks} (that is, finite groups with the conjugation operation $g \op h = h^{-1}gh$).

Using the finite quotients $\oG_{(X,r)}$ and other arguments, we improve the understanding of structure groups. In particular, we show that: 
\begin{itemize}
\item $G_{(X,r)}$ is \emph{virtually abelian}, hence \emph{linear}, hence \emph{residually finite};
\item the rank of its \emph{abelianization} is the number $k_r$ of its \emph{orbits} with respect to the actions $x \mapsto \sigma_y(x)$ and $x \mapsto \tau_y(x)$: \[\rk (\Ab G_{(X,r)}) =k_r;\] 
\item every solution shrinks to an injective one with the same structure group; this construction is completely algorithmic;
\item MP solutions are the only involutive solutions with \emph{diffuse} structure group; this answers a question of Chouraqui \cite{ChouOrd};
\item $G_{(X,r)} \text{ is biorderable } \; \Longleftrightarrow \; G_{(X,r)} \cong \Z^{k_r}$;
\item for the structure group of an SD solution, the following conditions are equivalent:
\[\text{ biorderable } \; \Longleftrightarrow \; \text{ left-orderable } \; \Longleftrightarrow \; \text{ (free) abelian } \; \Longleftrightarrow \; \text{ torsion free}.\] 
\end{itemize}

Left orderability and diffusion for general structure groups are more delicate problems, and remain open. Another open question is to understand the torsion of $\Ab G_{(X,r)}$ in terms of basic characteristics of $(X,r)$.

One more interesting question, not addressed in this paper, would be to determine what finite groups can be obtained as quotients of a given structure group $G_{(X,r)}$ by normal abelian subgroups. Is $G^0_{(X,r)}$ minimal in this family? Is our $\oG_{(X,r)}$ minimal among quotients preserving injectivity?

The paper is organized as follows. Sections \ref{s:StrGroups}--\ref{s:StrRacks} summarize basic definitions and results on structure groups and racks. They also contain new results:
\begin{itemize}
 \item The right and the left structure racks of a solution are isomorphic.
 \item The structure groups of a solution and of its biquandle quotient are isomorphic. This can be understood in terms of the structure rack and its quandle quotient. (\emph{Biquandles} and \emph{quandles} are, respectively, solutions and racks satisfying an additional condition, which simplifies the study of their structure groups.)
\end{itemize}  
Section~\ref{s:1cocycles} contains a technical lemma on passing to quotients in bijective (semi)group $1$-cocycles, illustrated with the $1$-cocycle $J \colon G_{(X,r)} \to G_{(X,\op_r)}$. In Section~\ref{s:QuotientsSol}, this lemma is used to pull back the finite quotient $\oG_{(X,\op_r)}$, described in Section~\ref{s:QuotientsRack}, to get the desired finite quotient $\oG_{(X,r)}$. Sections \ref{s:Applications}--\ref{s:Orderability} contain applications of our construction, and prove the properties of structure groups listed above. In Appendix~\ref{s:Size3}, we describe all size $3$ biquandles, their structure groups, and the finite quotients $\overline{G}$ thereof. Even these small examples exhibit a wide range of behaviours. Thus among solutions with the same structure rack we find solutions with different~$G$, or the same $G$ but different $\overline{G}$, or else the same $G$ and $\overline{G}$ but distinct $1$-cocycles $J$.

\section{Structure groups for YBE solutions}\label{s:StrGroups}

In this paper, by a \emph{solution} we mean a finite invertible non-degenerate set-theoretic solution $(X,r)$ of the \emph{Yang--Baxter equation} (YBE)
\[
r_{1}r_{2}r_{1}=r_{2}r_{1}r_{2}, \qquad \text{ where } r_{1}=r \times \Id_X, \,r_{2} = \Id_X \times r.
\]
This means that $X$ is a finite set, the map
\begin{align*}
r\colon X\times X &\to X\times X,\\
(x,y) &\mapsto (\sigma_x(y),\tau_y(x))
\end{align*}
is bijective, with the inverse
\begin{align*}
r^{-1} \colon 
(x,y) &\mapsto (\hsigma_x(y),\htau_y(x)),
\end{align*}
and its components $\sigma_x,\tau_x \colon X \to X$ are bijective for all $x \in X$. Such solutions are used for producing efficient coloring invariants of knots and their higher-dimensional analogues: see \cite{FRS_Species,NelVo,Rump_Biqu} and references therein. They are thus actively studied by knot theorists, who call them \emph{biracks}, and use the word \emph{biquandles} for biracks endowed with a bijection $t \colon X \to X$ satisfying 
\begin{align}\label{eqn:biquandle}
 &\forall x \in X, \qquad r(t(x),x) = (t(x),x).
\end{align}

To illustrate certain concepts and proofs, we will use graphical calculus, providing a bare minimum of explanations. More details can be found, for instance, in~\cite{LebVen}. The map $r$ is presented as on Fig.~\ref{fig:sol}. Our conditions on $r$ mean that any two neighboring colors on its diagram uniquely determine the two remaining colors. The YBE translates as the topological Reidemeister~$\mathrm{III}$ move. This graphical calculus provides a bridge towards knot theory. 
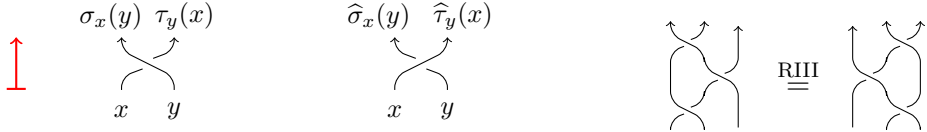
\begin{figure}[h]\centering
\begin{tikzpicture}[scale=0.7]
\draw [rounded corners](0,0)--(0,0.25)--(0.4,0.4);
\draw [rounded corners,->](0.6,0.6)--(1,0.75)--(1,1);
\draw [rounded corners,->](1,0)--(1,0.25)--(0,0.75)--(0,1);
\node  at (0,-0.4) {$x$};
\node  at (1,-0.4) {$y$};
\node  at (-.2,1.4) {$\sigma_x(y)$};
\node  at (1.2,1.4) {$\tau_y(x)$};
\draw [|->, red, thick]  (-2,0) -- (-2,1);
\end{tikzpicture} \hspace*{1.3cm}
\begin{tikzpicture}[scale=0.7]
\draw [rounded corners,->](0,0)--(0,0.25)--(1,0.75)--(1,1);
\draw [rounded corners](1,0)--(1,0.25)--(.6,.4);
\draw [rounded corners,->](.4,.6)--(0,0.75)--(0,1);
\node  at (0,-0.4) {$x$};
\node  at (1,-0.4) {$y$};
\node  at (-.3,1.4) {$\hsigma_x(y)$};
\node  at (1.3,1.4) {$\htau_y(x)$};
\end{tikzpicture} \hspace*{2cm}
\begin{tikzpicture}[xscale=0.45,yscale=0.45]
\draw [rounded corners](0,0)--(0,0.25)--(0.4,0.4);
\draw [rounded corners](0.6,0.6)--(1,0.75)--(1,1.25)--(1.4,1.4);
\draw [rounded corners,->](1.6,1.6)--(2,1.75)--(2,3);
\draw [rounded corners](1,0)--(1,0.25)--(0,0.75)--(0,2.25)--(0.4,2.4);
\draw [rounded corners,-<](0.6,2.6)--(1,2.75)--(1,3);
\draw [rounded corners,-<](2,0)--(2,1.25)--(1,1.75)--(1,2.25)--(0,2.75)--(0,3);
\node  at (4,1.5){\Large $\overset{\mathrm{RIII}}{=}$\hspace*{.2cm}};
\end{tikzpicture}
\begin{tikzpicture}[xscale=0.45,yscale=0.45]
\draw [rounded corners](1,1)--(1,1.25)--(1.4,1.4);
\draw [rounded corners,-<](1.6,1.6)--(2,1.75)--(2,3.25)--(1,3.75)--(1,4);
\draw [rounded corners](0,1)--(0,2.25)--(0.4,2.4);
\draw [rounded corners](0.6,2.6)--(1,2.75)--(1,3.25)--(1.4,3.4);
\draw [rounded corners,-<](1.6,3.6)--(2,3.75)--(2,4);
\draw [rounded corners,->](2,1)--(2,1.25)--(1,1.75)--(1,2.25)--(0,2.75)--(0,4);
\end{tikzpicture}
\caption{Crossings representing a solution, and the R$\mathrm{III}$ move representing the YBE.}\label{fig:sol}
\end{figure}

A (right) \emph{rack} is a set $X$ with a (right) self-distributive binary operation $\op$, 
in the sense of
\[(x \op y) \op z = (x \op z) \op (y \op z)\]
for all $x,y,z \in X$, such that the right translations 
\[\rho_y \colon x \mapsto x \op y\] 
are bijections $X \to X$ for all $y \in X$. A \emph{quandle} is a rack satisfying $x \op x = x$ for all $x \in X$. A rack $(X,\op)$ yields two solutions: $\Sol(X,\op)=(X,r_{\op})$ and $\Sol'(X,\op)=(X,r'_{\op})$, with
\[r_{\op}(x,y)=(y, x \op y), \qquad\qquad r'_{\op}(x,y)=(y \op x, x),\] 
called \emph{SD solutions}. We mainly work with $r_{\op}$ here, the properties of $r'_{\op}$ being similar. An SD solution $(X,r_{\op})$ is a biquandle if and only if $(X,\op)$ is a quandle, hence the terminology. There are also symmetric notions of a \emph{left rack} $(X,\lop)$, and \emph{left SD solutions} $(X,r_{\lop})$ and $(X,r'_{\lop})$, with $r_{\lop}(x,y)=(x \lop y, x)$, $r'_{\lop}(x,y)=(y, y \lop x)$. 

Involutive and SD solutions form the two best understood solution families. While involutive solutions are particularly interesting to algebraists, SD ones were, until recently, the realm of knot theorists. The structure rack construction recalled in Section~\ref{s:StrRacks} shows that these solutions also have an algebraic interest.

The \emph{structure group} of a solution $(X,r)$ is defined by generators and relations:
\[G_{(X,r)} = \langle X \,|\, xy = \sigma_x(y)\tau_y(x) \text{ for all } x,y \in X \rangle.\]
The structure group of a rack $(X,\op)$ is defined as the structure group of $(X,r_{\op})$ or $(X,r'_{\op})$, denoted by $G_{(X,\op)}$ or $G'_{(X,\op)}$ respectively. The map $x_1^{\varepsilon_1} \ldots x_s^{\varepsilon_s} \mapsto x_s^{\varepsilon_s} \ldots x_1^{\varepsilon_1}$, where $x_i \in X$, $\varepsilon_i = \pm 1$, induces a group isomorphism $G_{(X,\op)}^{op} \cong G'_{(X,\op)}$.

A solution is called \emph{injective} if the natural map
\begin{align*}
\iota \colon X &\to G_{(X,r)},\\
x &\mapsto x
\end{align*}
is injective. All involutive solutions are injective (see Section~\ref{s:1cocycles}). Conversely, many properties of involutive solutions generalize to injective ones. A simple example of a non-injective solution is $(\Z,r_{\op})$, where $\Z$ is considered as a rack, with $x \op y = x+1$ for all $x,y \in \Z$. Indeed, from $r_{\op}(x,x)=(x,x+1)$ one deduces $G_{(\Z,\op)} \cong (\Z,+)$, and the map $\iota \colon \Z \to \Z$ sends all $x \in \Z$ to $1$. A rack $(X,\op)$ is called \emph{injective} if the corresponding solution $(X,r_{\op})$ (or, equivalently, $(X,r'_{\op})$) is such.

If $(X,r)$ is a solution, then so is $(X,r^{-1})$. The YBE for $r$ and $r^{-1}$ implies
\begin{align}
\tau_y\tau_x &= \tau_{\tau_y(x)}\tau_{\sigma_x(y)}, & \htau_y\htau_x &= \htau_{\tau_y(x)} \htau_{\sigma_x(y)}, \label{eqn:tau_action}\\
\sigma_x\sigma_y &= \sigma_{\sigma_x(y)}\sigma_{\tau_y(x)}, & \hsigma_x\hsigma_y &= \hsigma_{\sigma_x(y)} \hsigma_{\tau_y(x)} \label{eqn:sigma_action}
\end{align}
for all $x,y \in X$. As a consequence, $\tau$ and $\htau$ induce right actions of $G_{(X,r)}$ on $X$, and $\sigma$ and $\hsigma$ induce left actions.

We will prove results on structure groups of racks and solutions while working mostly with (bi)quandles. The following construction makes it possible.

\begin{pro}\label{pro:birack_biquandle}
For a solution $(X,r)$, consider the smallest equivalence relation $\bumpeq$ such that $x \bumpeq \sigma_x(y)$ for all $x,y \in X$ satisfying $y=\tau_y(x)$. Then $r$ induces a biquandle structure $r'$ on $X/\!\bumpeq$. Moreover, the quotient map $X \twoheadrightarrow X/\!\bumpeq$ induces a group isomorphism
\[G_{(X,r)} \overset{\sim}{\longrightarrow} G_{(X/\!\bumpeq,\, r')}.\]
\end{pro}

\begin{defn}
The biquandle $(X/\!\bumpeq,\, r')$ from the proposition is called the \emph{induced biquandle} of $(X,r)$, denoted by $\BQ (X,r)$.
\end{defn}

The relation $\bumpeq$ is best understood diagrammatically: one identifies $x$ and $x'$ whenever the coloring situation from Fig.~\ref{fig:ind_biqu} occurs.

\begin{figure}[h]\centering
\begin{tikzpicture}[scale=0.7]
\draw [rounded corners](0,0)--(0,0.25)--(0.4,0.4);
\draw [rounded corners,->](0.6,0.6)--(1,0.75)--(1,1);
\draw [rounded corners,->](1,0)--(1,0.25)--(0,0.75)--(0,1);
\node  at (0,-0.4) {$x$};
\node  at (1,-0.4) {$y$};
\node  at (0,1) [above] {$x'$};
\node  at (1,1) [above] {$y$};
\end{tikzpicture} 
\caption{Identifying $x$ with $x'$ in all such situations, one shrinks a solution to its induced biquandle.}\label{fig:ind_biqu}
\end{figure}
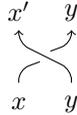

\begin{proof}
Take any $x,y \in X$ satisfying $y=\tau_y(x)$. Put $x'=\sigma_x(y)$. By Equations \eqref{eqn:tau_action}--\eqref{eqn:sigma_action}, one has $\tau_{x'} = \tau_x$ and $\sigma_{x'} = \sigma_x$, implying
\begin{align*}
r(x,z) &= (\sigma_x(z),\tau_z(x)), & r(z,x) &= (\sigma_z(x),\tau_x(z)),\\
r(x',z) &= (\sigma_x(z),\tau_z(x')), & r(z,x') &= (\sigma_z(x'),\tau_x(z))
\end{align*}
for any $z \in X$. To show that $r$ induces a map $r'$ on $(X/\!\bumpeq)^{\times 2}$, it suffices to check the relations $\tau_z(x) \bumpeq \tau_z(x')$, $\sigma_z(x) \bumpeq \sigma_z(x')$. In other terms, one needs to verify that the right and left $G_{(X,r)}$-actions on $X$ via $\tau$ and $\sigma$ respect the relation $\bumpeq$.

Put $r(z,x') = (\sigma_z(x'),w)$, and $r(w,y)=(u,v)$; precise expressions of $u,v,w$ in terms of $x,y,z$ are not important here. All these notations are summarized in Fig.~\ref{fig:SigmaRespectsBump}. Then 
\[(\sigma_z(x'),u,v) = r_2r_1r_2(z,x,y) = r_1r_2r_1(z,x,y),\]
which means $r(z,x) = (\sigma_z(x),\tilde{w})$, $r(\tilde{w},y)=(\tilde{u},v)$, and $r(\sigma_z(x),\tilde{u}) = (\sigma_z(x'), u)$ for some $\tilde{u},\tilde{w} \in X$. By the non-degeneracy of $r$, $r(w,y)=(u,v)$ and $r(\tilde{w},y)=(\tilde{u},v)$ imply $w=\tilde{w}$, $u= \tilde{u}$. But then $r(\sigma_z(x),u) = (\sigma_z(x'), u)$, hence $\sigma_z(x) \bumpeq \sigma_z(x')$. The proof of $\tau_z(x) \bumpeq \tau_z(x')$ is similar.

\begin{figure}[h]\centering
\begin{tikzpicture}[scale=0.8]
\draw [rounded corners](1,1)--(1,1.25)--(1.4,1.4);
\draw [rounded corners,->](1.6,1.6)--(2,1.75)--(2,3.25)--(1,3.75)--(1,4);
\draw [rounded corners](0,1)--(0,2.25)--(0.4,2.4);
\draw [rounded corners](0.6,2.6)--(1,2.75)--(1,3.25)--(1.4,3.4);
\draw [rounded corners,->](1.6,3.6)--(2,3.75)--(2,4);
\draw [rounded corners,->](2,1)--(2,1.25)--(1,1.75)--(1,2.25)--(0,2.75)--(0,4);
\node  at (5,2.5){\Large $\overset{\mathrm{RIII}}{\longleftrightarrow}$\hspace*{1cm}};
\node  at (0,.7) {$ z$};
\node  at (1,.7) {$ x$};
\node  at (2,.7) {$ y$};
\node  at (-.2,4.3) {$ \sigma_z(x')$};
\node  at (1,4.3) {$ u$};
\node  at (2,4.3) {$ v$};
\node  at (1.3,2) {$ x'$};
\node  at (2.2,1.9) {$ y$};
\node  at (.8,3) {$ w$};
\end{tikzpicture}
\begin{tikzpicture}[scale=0.8]
\draw [rounded corners](0,0)--(0,0.25)--(0.4,0.4);
\draw [rounded corners](0.6,0.6)--(1,0.75)--(1,1.25)--(1.4,1.4);
\draw [rounded corners,->](1.6,1.6)--(2,1.75)--(2,3);
\draw [rounded corners](1,0)--(1,0.25)--(0,0.75)--(0,2.25)--(0.4,2.4);
\draw [rounded corners,->](0.6,2.6)--(1,2.75)--(1,3);
\draw [rounded corners,->](2,0)--(2,1.25)--(1,1.75)--(1,2.25)--(0,2.75)--(0,3);
\node  at (0,-.3) {$ z$};
\node  at (1,-.3) {$ x$};
\node  at (2,-.3) {$ y$};
\node  at (-.2,3.3) {$ \sigma_z(x')$};
\node  at (1,3.3) {$ u$};
\node  at (2,3.3) {$ v$};
\node  at (-.6,1.4) {$ \sigma_z(x)$};
\node  at (1.2,2) {$ u$};
\node  at (1.2,1) {$ w$};
\end{tikzpicture}
\caption{Relation $x \bumpeq x'$ implies $\sigma_z(x) \bumpeq \sigma_z(x')$ for all $z$.}\label{fig:SigmaRespectsBump}
\end{figure}
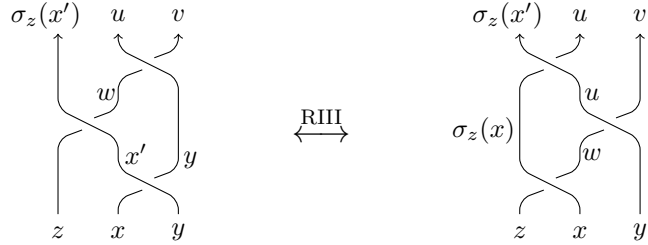

The YBE and the surjectivity for $r'$, as well as the surjectivity of its $\sigma'$- and $\tau'$-components, follow from the corresponding properties for $r$. Since the set $X/\!\bumpeq$ is finite, surjectivity implies bijectivity. So, $(X/\!\bumpeq,r')$ is indeed a solution.

For any $y \in X$, the element $t(y):={\tau_y}^{-1}(y)$ satisfies $y=\tau_y(t(y))$. The definition of the relation~$\bumpeq$ then yields $t(y) \bumpeq \sigma_{t(y)}(y)$. Hence the map $t' \colon X/\!\bumpeq \,\to X/\!\bumpeq$ induced by $t$ satisfies $r'(t'(y),y) = (t'(y),y)$ for all $y \in X/\!\bumpeq$. This map is injective, since $y$ reconstructs from $t'(y)$ via $y={\sigma_{t'(y)}}^{-1}(t'(y))$. By the finiteness of $X/\!\bumpeq$, $t'$ is then bijective, and $(X/\!\bumpeq,r')$ is a biquandle.

Finally, for all $x,y \in X$ satisfying $y=\tau_y(x)$, from $xy = \sigma_x(y)\tau_y(x)$ one deduces $x=\sigma_x(y)$ in $G_{(X,r)}$, so when passing to the quotient $X/\!\bumpeq$, one does not change the structure group.
\end{proof}

\begin{rem}\label{rem:IndBiquUniv}
Induced biquandles enjoy the following universal property: for any morphism $\phi \colon X \to Y$ of solutions, where $Y$ is a biquandle, there is a unique morphism $\phi' \colon X/\!\bumpeq\, \to Y$ with $\phi = \phi'\pi$. Here $\pi$ is the quotient map $X \twoheadrightarrow X/\!\bumpeq$.
\end{rem}

Arguments from the proof of Proposition~\ref{pro:birack_biquandle} imply the following elementary observation:

\begin{lem}\label{lem:Inj_Biqu}
An injective solution is necessarily a biquandle.
\end{lem}

Our induced biquandle construction in fact defines a functor $\BQ \colon \SolCat \to \BQCat$, where the category $\SolCat$ (resp. $\BQCat$) of solutions (resp. biquandles) and their morphisms is defined in the obvious way. This functor is a retraction for the inclusion functor $\BQCat \to \SolCat$. Indeed, if $(X,r)$ is already a biquandle, then $\BQ (X,r) = (X,r)$.

For a solution $(X,r_{\op})$ associated with a rack, the quotient $X/\!\bumpeq$ simply identifies $x$ with $x \op x$ for all $x \in X$. The operation $\op$ induces a quandle operation $\op'$ on $X/\!\bumpeq$. The quandle $\Q(X,\op) = (X/\!\bumpeq,\, \op')$ is the \emph{induced quandle} of $(X,\op)$. It appeared in \cite{Brieskorn,AndrGr}. At the level of solutions and structure groups, one has
\begin{align*}
\BQ(\Sol(X,\op)) &\cong \Sol(\Q(X,\op)),\\
G_{(X,\op)}  &\cong G_{\Q(X,\op)}. 
\end{align*}
The functoriality of all our constructions allows us to lift these identities to relations between the functors $\BQ$, $\Q \colon \Rack \to \Quandle$, $\Sol \colon \Rack \to \SolCat$ and its restriction $\Sol' \colon \Quandle \to \BQCat$. The last two functors send a rack or quandle $(X,\op)$ to $(X,r_{\op})$.

\begin{pro}\label{pro:commdiag_Sol_Q_BQ}
The following functors assemble into a commutative diagram: 
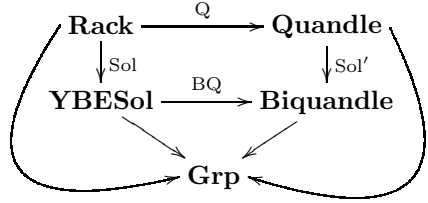
\[\xymatrix@!0 @R=1cm @C=1.5cm{
\Rack \ar[rr]^-{\Q} \ar[d]^{\Sol} \ar@/_4pc/[]!<-3ex,1ex>;[ddr]!<-3ex,0ex> && \Quandle \ar[d]^{\Sol'} \ar@/^4pc/[]!<5ex,1ex>;[ddl]!<3ex,0ex>\\
\SolCat \ar[rr]^-{\BQ} \ar[dr] && \BQCat \ar[dl]\\
&\Grp &}\]
Here all unlabelled arrows correspond to the structure group functors.
\end{pro}

\section{Structure racks for YBE solutions}\label{s:StrRacks}

We now describe retractions for the inclusion functors $\Sol$ above. As explained in \cite{Sol,LYZ,LebVen}, to any solution $(X,r)$ one can associate its (right) \emph{structure rack} $\R(X,r) = (X,\op_r)$, where 
\[x \op_r y = \tau_y \sigma_{\tau_x^{-1}(y)}(x) = \tau_y \htau_y^{-1}(x)\]
for all $x,y \in X$. Recall that $\htau$ is the right component of the map $r^{-1}$. In the literature, $\R(X,r)$ is also called the \emph{associated} or \emph{derived rack} of $(X,r)$. It has a symmetric left version $(X,\lop_r)$, which we will show to be isomorphic to $\R(X,r)$ (Proposition~\ref{pro:LeftRightStrRack}). The graphical definition of structure racks from Fig.~\ref{fig:str_racks} makes this construction more intuitive.
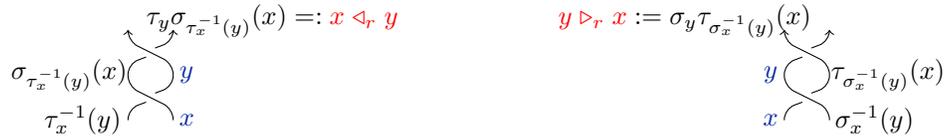
\begin{figure}[h]\centering
\begin{tikzpicture}[scale=0.6]
\draw [rounded corners](0,0)--(0,0.25)--(0.4,0.4);
\draw [->, rounded corners](0.6,0.6)--(1,0.75)--(1,1.25)--(0,1.75)--(0,2);
\draw [rounded corners](1,0)--(1,0.25)--(0,0.75)--(0,1.25)--(0.4,1.4);
\draw [->, rounded corners](0.6,1.6)--(1,1.75)--(1,2);
\node  at (1.3,0){$\color{myblue} x$};
\node  at (1.3,1){$\color{myblue} y$};
\node  at (3.2,2.2){$\tau_y \sigma_{\tau_x^{-1}(y)}(x) =: \color{red} x \op_r y$};
\node  at (-1,0){$\tau_x^{-1}(y)$};
\node  at (-1.3,1){$\sigma_{\tau_x^{-1}(y)}(x)$};
\end{tikzpicture}\hspace*{1.8cm}
\begin{tikzpicture}[scale=0.6]
\draw [rounded corners](0,0)--(0,0.25)--(0.4,0.4);
\draw [->, rounded corners](0.6,0.6)--(1,0.75)--(1,1.25)--(0,1.75)--(0,2);
\draw [rounded corners](1,0)--(1,0.25)--(0,0.75)--(0,1.25)--(0.4,1.4);
\draw [->, rounded corners](0.6,1.6)--(1,1.75)--(1,2);
\node  at (-0.3,0){$\color{myblue} x$};
\node  at (-0.3,1){$\color{myblue} y$};
\node  at (-2.2,2.2){$\color{red} y \lop_r x \color{black}:= \sigma_y \tau_{\sigma_x^{-1}(y)}(x)$};
\node  at (2,0){$\sigma_x^{-1}(y)$};
\node  at (2.3,1){$\tau_{\sigma_x^{-1}(y)}(x)$};
\end{tikzpicture}
\caption{The colors $x,y$ uniquely determine all colors in both diagrams; the upper right/left color defines the right/left structure rack of a solution.}\label{fig:str_racks}
\end{figure} 

The structure rack captures basic properties of the original solution:
\begin{itemize}
\item $(X,r)$ is involutive $\; \Longleftrightarrow \; (X,\op_r)$ is trivial: $x \op_r y = x$ for all $x,y \in X$;
\item $(X,r)$ is a biquandle $\; \Longleftrightarrow \; (X,\op_r)$ is a quandle;
\item the actions of the braid groups $B_n$ on the tensor powers $X^{\times n}$ induced by $r$ and by $\op_r$ are isomorphic (without the solutions $(X,r)$ and $(X,r_{\op_r})$ being necessarily isomorphic).
\end{itemize}

The structure rack construction yields a functor $\R \colon \SolCat \to \Rack$, which is a left inverse of the functor $\Sol \colon \Rack \to \SolCat$. They restrict to functors $\R' : \BQCat \rightleftarrows \Quandle : \Sol'$, satisfying $\R' \circ \Sol' = \Id_{\Quandle}$.

We next show that the functors $\R$ and $\R'$ intertwine the functors $\Q$ and $\BQ$. Some technical results are first due:

\begin{lem}\label{lem:SqProperties}
Let $(X,r)$ be a solution. The maps
\begin{align*}
T \colon X &\to X,& \text{ and }\qquad\qquad U \colon X &\to X,\\
y &\mapsto \tau_y^{-1} (y);  & x &\mapsto \sigma_x^{-1}(x \lop_r x)
\end{align*}
are mutually inverse. So are the maps
\begin{align*}
T^- \colon X &\to X,& \text{ and }\qquad\qquad U^- \colon X &\to X,\\
y &\mapsto \sigma_y^{-1} (y); & x &\mapsto \tau_x^{-1}(x \op_r x).
\end{align*}
\end{lem}

In other words, for any $x,y \in X$ the following equivalences hold:
\begin{align*}
\tau_y(x) = y \; &\Longleftrightarrow \; \sigma_x(y) = x \lop_r x,\\
\sigma_x(y) = x \; &\Longleftrightarrow \; \tau_y(x) = y \op_r y.
\end{align*} 

\begin{proof}
We will consider only the first pair of maps. The second one can be treated in a symmetric way; graphically, this is the vertical mirror symmetry (cf. Fig.~\ref{fig:SqProperties}).

Suppose that $x= T(y)$, that is, $\tau_y(x) = y$. Then $\hsigma_x= \hsigma_{\sigma_x(y)}$ by~\eqref{eqn:sigma_action}, and 
\[\hsigma_x (y)= \hsigma_{\sigma_x(y)} (y) = \hsigma_{\sigma_x(y)} (\tau_y(x)) = x,\]
since the map $(x,y) \mapsto (\hsigma_x(y),\htau_y(x))$ is the inverse of $(x,y) \mapsto (\sigma_x(y),\tau_y(x))$. This yields $\sigma_x(y) = x \lop_r x$, as shown in Fig.~\ref{fig:SqProperties}. In other words, $y=U(x)$.
\begin{figure}[h]\centering
\begin{tikzpicture}[xscale=.7,yscale=.6]
\draw [rounded corners](0,-1)--(0,-.75)--(0.4,-.6);
\draw [rounded corners](0.6,-.4)--(1,-.25)--(1,0);
\draw [rounded corners](1,-1)--(1,-.75)--(0,-.25)--(0,0);
\draw [rounded corners](0,0)--(0,0.25)--(0.4,0.4);
\draw [rounded corners,->](0.6,0.6)--(1,0.75)--(1,1);
\draw [rounded corners,->](1,0)--(1,0.25)--(0,0.75)--(0,1);
\node  at (-.2,0) {$\color{myblue} x$};
\node  at (1.2,0) {$\color{myblue} y$};
\node  at (-1.2,1.4) {${\color{red} x \lop_r x = \,} \sigma_x(y)$};
\node  at (1.2,1.4) {$\color{myblue} y$};
\node  at (-.2,-1) {$x$};
\end{tikzpicture} \hspace*{2cm}
\begin{tikzpicture}[xscale=.7,yscale=.6]
\draw [rounded corners](0,-1)--(0,-.75)--(0.4,-.6);
\draw [rounded corners](0.6,-.4)--(1,-.25)--(1,0);
\draw [rounded corners](1,-1)--(1,-.75)--(0,-.25)--(0,0);
\draw [rounded corners](0,0)--(0,0.25)--(0.4,0.4);
\draw [rounded corners,->](0.6,0.6)--(1,0.75)--(1,1);
\draw [rounded corners,->](1,0)--(1,0.25)--(0,0.75)--(0,1);
\node  at (-.2,0) {$\color{myblue} x$};
\node  at (1.2,0) {$\color{myblue} y$};
\node  at (2.2,1.4) {$\tau_y(x) {= \color{red} y \op_r y}$};
\node  at (-.2,1.4) {$\color{myblue} x$};
\node  at (1.2,-1) {$y$};
\end{tikzpicture}
\caption{Left: $\tau_y(x) = y$ implies $\sigma_x(y) = x \lop_r x$. Right: $\sigma_x(y) = x$ implies $\tau_y(x) = y \op_r y$.}\label{fig:SqProperties}
\end{figure}

Now, given $y=U(x)$, we shall deduce $x= T(y)$. Rewrite $y=U(x)=\sigma_x^{-1}(x \lop_r x)$ as $\sigma_x(y)=x \lop_r x$, hence $r(x,y)=(x \lop_r x,\tau_y(x))$. From the definition of $\lop_r$, we deduce $r^{-1}(x,y)=(x,z)$, where $z$ satisfies $\tau_z(x)=y$. Again, \eqref{eqn:tau_action} implies $\tau_y=\tau_z$, so $\tau_y(x)=\tau_z(x)=y$. Thus $r(x,y)=(x \lop_r x,y)$, hence $x=\tau_y^{-1} (y)= T(y)$.
\end{proof}

\begin{exa}
If $(X,r)$ is a biquandle, then $T$ is the map $t$ defined by~\eqref{eqn:biquandle}.
\end{exa}

\begin{exa}
Let $(X,\op)$ be a rack. For the solution $(X,r_{\op})$, one has $T=\Sq^{-1}$, where the map $\Sq$ is defined in Lemma~\ref{lem:SqPropertiesRack}. For $(X,r'_{\op})$, the $T$-map is just $\Id_X$.
\end{exa}

\begin{lem}\label{lem:SqLeftRight}
Let $(X,r)$ be a solution. For any $x \in X$, one has $x \op_r x = x \lop_r x$.
\end{lem}

\begin{proof} According to Lemma~\ref{lem:SqProperties}, $\forall x \in X, \, \exists ! y \in X$ satisfying $r(x,y) = (x \lop_r x,\, y)$. Put $z=\tau_{x \lop_r x}^{-1}(x)$. By the YBE, one has
\[r_2r_2r_1r_2(z,x,y) = r_2r_1r_2r_1(z,x,y) = r_1r_2r_1r_1(z,x,y).\]
Let us study the three expressions in more detail. First,
\[(z,x,y) \overset{r_2}{\mapsto} (z,x \lop_r x,y) \overset{r_1}{\mapsto} (u,x,y) \overset{r_2}{\mapsto} (u,x \lop_r x,y) \overset{r_2}{\mapsto} (u,x \lop_r x,y \op_r y),\]
where $u=\sigma_z(x \lop_r x)$, and we used the consequence $\sigma_x=\sigma_{x \lop_r x}$ of~\eqref{eqn:sigma_action} to get $r(x \lop_r x,y) = (x \lop_r x,y \op_r y)$. Further, $r_1r_2r_1(z,x,y) = r_2r_1r_2(z,x,y) = (u,x \lop_r x,y)$, so
\[(z,x,y) \overset{r_1}{\mapsto} (v,x',y) \overset{r_2}{\mapsto} (v,w,y) \overset{r_1}{\mapsto} (u,x \lop_r x,y) \overset{r_2}{\mapsto} (u,x \lop_r x,y \op_r y)\]
for certain $x',v,w \in X$. Since $r$ is non-degenerate, $r(x',y) = (w,y)$ implies $x'=x$ and $w=x \lop_r x$.  Finally, 
\[(z,x,y) \overset{r_1}{\mapsto} (v,x,y) \overset{r_1}{\mapsto} (s,x \op_r x,y) \overset{r_2}{\mapsto} (s,t,y \op_r y) \overset{r_1}{\mapsto} (u,x \lop_r x,y \op_r y)\]
for certain $s,t \in X$. The key relation here is $r(x \op_r x,y) = (t,y \op_r y)$. Together with $r(x \lop_r x,y) = (x \lop_r x,y \op_r y)$ and the non-degeneracy of~$r$, it implies $x \op_r x = x \lop_r x$, as desired.

These computations are summarized in Fig.~\ref{fig:SqProperties2}. 
\begin{figure}[h]\centering
\begin{tikzpicture}[scale=1]
\draw [rounded corners](1,1)--(1,1.25)--(1.4,1.4);
\draw [rounded corners](1.6,1.6)--(2,1.75)--(2,3.25)--(1,3.75)--(1,4.25)--(1.4,4.4);
\draw [rounded corners,->](1.6,4.6)--(2,4.75)--(2,5);
\draw [rounded corners](0,1)--(0,2.25)--(0.4,2.4);
\draw [rounded corners](0.6,2.6)--(1,2.75)--(1,3.25)--(1.4,3.4);
\draw [rounded corners,->](1.6,3.6)--(2,3.75)--(2,4.25)--(1,4.75)--(1,5);
\draw [rounded corners,->](2,1)--(2,1.25)--(1,1.75)--(1,2.25)--(0,2.75)--(0,5);
\node  at (0,.8) {$z$};
\node  at (1,.8) {$x$};
\node  at (2,.8) {$\color{myblue} y$};
\node  at (2.2,2.5) {$\color{myblue} y$};
\node  at (2.2,4) {$\color{myblue} y$};
\node  [fill=white] at (1,4) {$\scriptstyle x \lop x$};
\node  [fill=white] at (1,3) {$x$};
\node  [fill=white] at (1,2) {$\scriptstyle x \lop x$};
\node  at (0,5.2) {$u$};
\node  at (1,5.2) {$x \lop x$};
\node  at (2.2,5.2) {$y \op y$};
\node  at (3.2,3){\Large $\leftrightarrow$\hspace*{.2cm}};
\end{tikzpicture}
\begin{tikzpicture}[scale=1]
\draw [rounded corners](0,0)--(0,0.25)--(0.4,0.4);
\draw [rounded corners](0.6,0.6)--(1,0.75)--(1,1.25)--(1.4,1.4);
\draw [rounded corners,->](1.6,1.6)--(2,1.75)--(2,3.25)--(1,3.75)--(1,4);
\draw [rounded corners](1,0)--(1,0.25)--(0,0.75)--(0,2.25)--(0.4,2.4);
\draw [rounded corners](0.6,2.6)--(1,2.75)--(1,3.25)--(1.4,3.4);
\draw [rounded corners,->](1.6,3.6)--(2,3.75)--(2,4);
\draw [rounded corners,->](2,0)--(2,1.25)--(1,1.75)--(1,2.25)--(0,2.75)--(0,4);
\node  at (3.2,2){\Large $\leftrightarrow$\hspace*{.2cm}};
\node  at (0,-.2) {$z$};
\node  at (1,-.2) {$x$};
\node  at (2,-.2) {$y$};
\node  at (-.2,1.5) {$v$};
\node  [fill=white] at (1,1) {$x$};
\node  [fill=white] at (1,2) {$\scriptstyle x \lop x$};
\node  [fill=white] at (1,3) {$\scriptstyle x \lop x$};
\node  at (2.2,3) {$y$};
\node  at (0,4.2) {$u$};
\node  at (1,4.2) {$x \lop x$};
\node  at (2.2,4.2) {$y \op y$};
\fill [myred!20, opacity=.2] (0.5,2.8) rectangle (2.8,4.5);
\end{tikzpicture}
\begin{tikzpicture}[scale=1]
\draw [rounded corners](1,-1)--(1,-.75)--(0,-.25)--(0,0.25)--(0.4,0.4);
\draw [rounded corners](0.6,0.6)--(1,0.75)--(1,1.25)--(1.4,1.4);
\draw [rounded corners,->](1.6,1.6)--(2,1.75)--(2,3);
\draw [rounded corners](0,-1)--(0,-.75)--(0.4,-.6);
\draw [rounded corners](0.6,-.4)--(1,-.25)--(1,0.25)--(0,0.75)--(0,2.25)--(0.4,2.4);
\draw [rounded corners,->](0.6,2.6)--(1,2.75)--(1,3);
\draw [rounded corners,->](2,-1)--(2,1.25)--(1,1.75)--(1,2.25)--(0,2.75)--(0,3);
\node  at (0,-1.2) {$z$};
\node  at (1,-1.2) {$x$};
\node  at (2,-1.2) {$y$};
\node  at (-.2,0) {$v$};
\node  at (-.2,1.5) {$s$};
\node  [fill=white] at (1,0) {$x$};
\node  [fill=white] at (1,1) {$\scriptstyle x \op x$};
\node  [fill=white] at (1,2) {$t$};
\node  at (2.2,1) {$y$};
\node  at (2.5,2) {$y \op y$};
\node  at (0,3.2) {$u$};
\node  at (1,3.2) {$x \lop x$};
\node  at (2.2,3.2) {$y \op y$};
\fill [myred!20, opacity=.2] (0.5,.8) rectangle (3,2.2);
\end{tikzpicture}
\caption{The three colors $y,y,y$ in the left diagram uniquely determine all colors in all the diagrams. Comparing the two highlighted crossings, one obtains $x \op_r x = x \lop_r x$. Here $z=\tau_{x \lop_r x}^{-1}(x)$.}\label{fig:SqProperties2}
\end{figure}
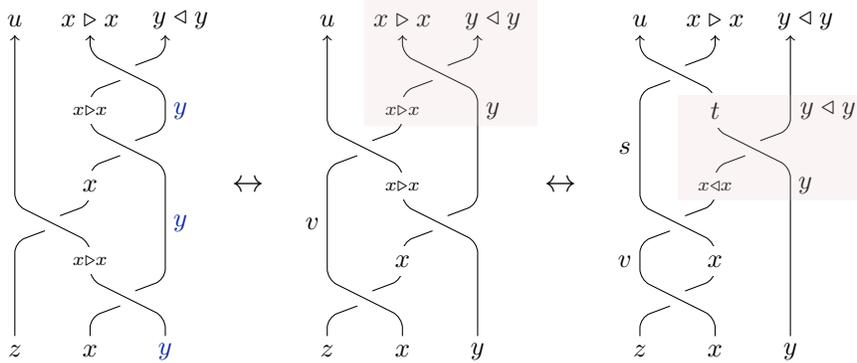
\end{proof}

\begin{lem}\label{lem:SqPropertiesRack}
Given a rack $(X,\op)$, consider the map $\Sq \colon X \to X, \, x \mapsto x \op x$. For any $x,y \in X$, one has:
\begin{enumerate}
 \item $\Sq(x \op y) = \Sq(x)\op y$;
 \item $x \op \Sq(y) = x \op y$;
 \item the map $\Sq$ is bijective.
\end{enumerate}
\end{lem}

More properties of this map can be found in \cite{AndrGr,Szymik}.

\begin{proof}
\begin{enumerate}
\item By self-distributivity, $(x \op y) \op (x \op y) = (x \op x) \op y$.
\item Since the right translation $\bullet \mapsto \bullet \op y$ is bijective, there exists a unique $z \in X$ satisfying $z \op y = x$. Then 
\[x \op (y \op y) = (z \op y) \op (y \op y) = (z \op y) \op y = x \op y.\]
\item Again by the bijectivity of right translations, for any $x \in X$ there exists a unique $z \in X$ with $z \op x = x$. Then 
\[\Sq(z) \op x = \Sq(z \op x)= \Sq(x) = x \op x,\]
implying $\Sq(z) = x$. So, the map $\Sq$ is surjective. Further, if $\Sq(z) = \Sq(z')$ for some $z,z' \in X$, then by the preceding point, for all $u \in X$ one has
\[u \op z = u \op \Sq(z) = u \op \Sq(z') = u \op z',\]
implying $z \op z = z' \op z' = z' \op z$, hence $z = z'$. \qedhere
\end{enumerate}
\end{proof}

\begin{notation}
Given a solution $(X,r)$, we will write $\Sq(x) = x \op_r x = x \lop_r x$.
\end{notation}

From the three lemmas, one obtains the coloring pattern from Fig.~\ref{fig:tower}.
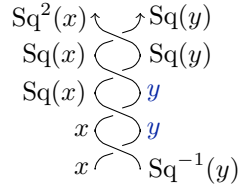
\begin{figure}[h]\centering
\begin{tikzpicture}[xscale=0.6,yscale=.5]
\foreach \i in {0,...,2}{
\draw [rounded corners](0,\i)--(0,\i+0.25)--(0.4,\i+0.4);
\draw [rounded corners](0.6,\i+0.6)--(1,\i+0.75)--(1,\i+1);
\draw [rounded corners](1,\i)--(1,\i+0.25)--(0,\i+0.75)--(0,\i+1);}
\foreach \i in {3}{
\draw [rounded corners](0,\i)--(0,\i+0.25)--(0.4,\i+0.4);
\draw [rounded corners,-<](0.6,\i+0.6)--(1,\i+0.75)--(1,\i+1);
\draw [rounded corners,-<](1,\i)--(1,\i+0.25)--(0,\i+0.75)--(0,\i+1);}
\node  at (2.2,0){$\Sq^{-1}(y)$};
\node [myblue] at (1.3,1){${y}$};
\node [myblue] at (1.3,2){${y}$};
\node  at (1.9,3){$\Sq(y)$};
\node  at (1.9,4){$\Sq(y)$};
\node at (-.3,1){$x$};
\node at (-.3,0){$x$};
\node  at (-.9,3){$\Sq(x)$};
\node  at (-.9,2){$\Sq(x)$};
\node  at (-1,4){$\Sq^2(x)$};
\end{tikzpicture}
\caption{This coloring pattern applies for any $y \in X$ and $x=T(y)$.}\label{fig:tower}
\end{figure} 

One also immediately deduces two important corollaries: 

\begin{pro}\label{pro:IndBiquandleViaSq}
Let $(X,r)$ be a solution. The equivalence relation~$\bumpeq$ on~$X$ defined in Proposition~\ref{pro:birack_biquandle} has the following alternative description:
\[x \bumpeq y\qquad \Longleftrightarrow \qquad y = \Sq^m(x) \, \text{ for some } m \in \Z.\]
\end{pro}


\begin{pro}\label{pro:commdiag_R_Q_BQ}
The following functors assemble into a commutative diagram:
\[\xymatrix@!0 @R=1cm @C=1.5cm{
\Rack \ar[rr]^-{\Q} && \Quandle\\
\SolCat \ar[rr]^-{\BQ} \ar[u]^{\R}  && \BQCat \ar[u]^{\R'} }\]
\end{pro}

In contrast to the induced (bi)quandle functor, the structure rack functor does not yield isomorphisms at the structure group level. Instead, it gives bijective group $1$-cocycles, as we will see in the next section.

We finish this section with a comparison of right and left structure racks. It is not used in what follows, but makes our structure rack story more coherent.

\begin{pro}\label{pro:LeftRightStrRack}
The right and the left structure racks of a solution $(X,r)$ are isomorphic. Concretely, the relation
\begin{align}
&T(y \lop_r x) = T(x) \op_r T(y) \label{eq:LR}
\end{align}
holds for all $x,y \in X$. Here $T$ is the bijection $y \mapsto \tau_y^{-1} (y)$ from Lemma~\ref{lem:SqProperties}.
\end{pro}


\begin{proof}
Recall that the structure rack operations can be written as
\begin{align*}
&x \op_r y = \tau_y \sigma_{\tau_x^{-1}(y)}(x) = \tau_y \htau_y^{-1}(x), && y \lop_r x = \sigma_y \tau_{\sigma_x^{-1}(y)}(x) = \sigma_y \hsigma_y^{-1}(x). \label{eq:LRop}
\end{align*}
We will also use two properties of the map $\tau_y$:
\begin{itemize}
 \item it is an automorphism of the rack $(X,\op_r)$ (Lemma~\ref{lem:TauActsByRackMor});
 \item it induces a right action of $G_{(X,r)}$ on $X$ (relation~\eqref{eqn:tau_action}).
\end{itemize}
With this in mind, the two sides of~\eqref{eq:LR} rewrite as
\begin{align*}
T(y \lop_r x) &= T \sigma_y \hsigma_y^{-1}(x),\\
T(x) \op_r T(y )&=\tau_y^{-1}(\tau_y T(x) \op_r \tau_y T(y)) = \tau_y^{-1}(\tau_y T(x) \op_r y) = \tau_y^{-1} \tau_y \htau_y^{-1} \tau_y T(x)\\ & =\htau_y^{-1} \tau_y T(x).
\end{align*}
It remains to check the relation 
\[T \sigma_y \hsigma_y^{-1} = \htau_y^{-1} \tau_y T\]
in $\Sym(X)$. We split it into two parts:
\begin{align*}
T \sigma_y &= \htau_y^{-1}T, & T \hsigma_y &= \tau_y^{-1}T.
\end{align*}
They are verified as follows:
\begin{align*}
T \sigma_y(x) &= \tau_{\sigma_y(x)}^{-1}\sigma_y(x) = \tau_{\sigma_y(x)}^{-1}\tau_{\tau_x(y)}^{-1} \tau_{\tau_x(y)} \sigma_y(x)\\
&=\tau_y^{-1} \tau_x^{-1}(x \op_r \tau_x(y)) = \tau_y^{-1}(T(x) \op_r y) = \htau_y^{-1}T(x);\\
T \hsigma_y(x) &= \tau_{\hsigma_y(x)}^{-1}\hsigma_y(x) = \tau_{\hsigma_y(x)}^{-1}\tau_{\htau_x(y)}^{-1} \tau_{\htau_x(y)} \hsigma_y(x)\\
&=\tau_y^{-1} \tau_x^{-1} \tau_{\htau_x(y)} \hsigma_y(x) = \tau_y^{-1} \tau_x^{-1} (x) = \tau_y^{-1}T(x).
\end{align*}
For a more enlightening graphical proof, see Figs \ref{fig:T_intertwines} and \ref{fig:T_intertwines2}.
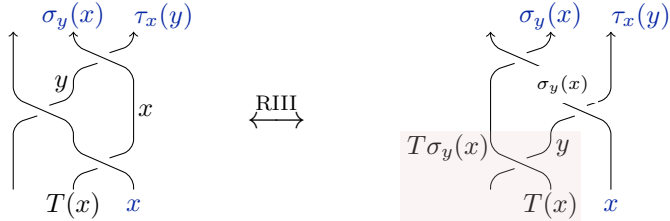
\begin{figure}[h]\centering
\begin{tikzpicture}[xscale=.8,yscale=.7]
\draw [rounded corners](1,1)--(1,1.25)--(1.4,1.4);
\draw [rounded corners,->](1.6,1.6)--(2,1.75)--(2,3.25)--(1,3.75)--(1,4);
\draw [rounded corners](0,1)--(0,2.25)--(0.4,2.4);
\draw [rounded corners](0.6,2.6)--(1,2.75)--(1,3.25)--(1.4,3.4);
\draw [rounded corners,->](1.6,3.6)--(2,3.75)--(2,4);
\draw [rounded corners,->](2,1)--(2,1.25)--(1,1.75)--(1,2.25)--(0,2.75)--(0,4);
\node  at (5,2.5){\Large $\overset{\mathrm{RIII}}{\longleftrightarrow}$\hspace*{1cm}};
\node  at (1,.7) {$T(x)$};
\node  at (2,.7) {$\color{myblue} x$};
\node  at (1,4.3) {$\color{myblue} \sigma_y(x)$};
\node  at (2.5,4.3) {$\color{myblue} \tau_x(y)$};
\node  at (2.2,2.5) {$x$};
\node  at (.8,3) {$y$};
\end{tikzpicture}
\begin{tikzpicture}[xscale=.8,yscale=.7]
\draw [rounded corners](0,0)--(0,0.25)--(0.4,0.4);
\draw [rounded corners](0.6,0.6)--(1,0.75)--(1,1.25)--(1.4,1.4);
\draw [rounded corners,->](1.6,1.6)--(2,1.75)--(2,3);
\draw [rounded corners](1,0)--(1,0.25)--(0,0.75)--(0,2.25)--(0.4,2.4);
\draw [rounded corners,->](0.6,2.6)--(1,2.75)--(1,3);
\draw [rounded corners,->](2,0)--(2,1.25)--(1,1.75)--(1,2.25)--(0,2.75)--(0,3);
\node  at (1,-.3) {$T(x)$};
\node  at (2,-.3) {$\color{myblue} x$};
\node  at (1,3.3) {$\color{myblue} \sigma_y(x)$};
\node  at (2.5,3.3) {$\color{myblue} \tau_x(y)$};
\node  [fill=white] at (1.2,2) {$\scriptstyle \sigma_y(x)$};
\node  at (-.7,.8) {$T \sigma_y(x)$};
\node  at (1.2,.8) {$y$};
\fill [myred!20, opacity=.2] (-1.5,-.6) rectangle (1.5,1.1);
\end{tikzpicture}
\caption{At the highlighted crossing, one gets $\htau_yT \sigma_y = T$. Here the two rightmost top colors and the rightmost bottom color uniquely determine the remaining ones for both diagrams.}\label{fig:T_intertwines}
\end{figure}
\begin{figure}[h]\centering
\begin{tikzpicture}[xscale=.8,yscale=.7]
\draw [rounded corners](1,1)--(1,1.25)--(1.4,1.4);
\draw [rounded corners](1.6,1.6)--(2,1.75)--(2,3.25)--(1.6,3.4);
\draw [rounded corners,->](1.4,3.6)--(1,3.75)--(1,4);
\draw [rounded corners,->](0,1)--(0,2.25)--(1,2.75)--(1,3.25)--(2,3.75)--(2,4);
\draw [rounded corners](2,1)--(2,1.25)--(1,1.75)--(1,2.25)--(.6,2.4);
\draw [rounded corners,->](.4,2.6)--(0,2.75)--(0,4);
\node  at (5,2.5){\Large $\overset{\mathrm{RIII}}{\longleftrightarrow}$\hspace*{1cm}};
\node  at (1,.7) {$T(x)$};
\node  at (2,.7) {$\color{myblue} x$};
\node  at (1,4.3) {$\color{myblue} \hsigma_y(x)$};
\node  at (2.5,4.3) {$\color{myblue} \htau_x(y)$};
\node  at (2.2,2.5) {$x$};
\node  at (.8,3) {$y$};
\end{tikzpicture}
\begin{tikzpicture}[xscale=.8,yscale=.7]
\draw [rounded corners,->](0,0)--(0,0.25)--(1,0.75)--(1,1.25)--(2,1.75)--(2,3);
\draw [rounded corners](1,0)--(1,0.25)--(.6,.4);
\draw [rounded corners](.4,.6)--(0,0.75)--(0,2.25)--(0.4,2.4);
\draw [rounded corners,->](0.6,2.6)--(1,2.75)--(1,3);
\draw [rounded corners](2,0)--(2,1.25)--(1.6,1.4);
\draw [rounded corners,->](1.4,1.6)--(1,1.75)--(1,2.25)--(0,2.75)--(0,3);
\node  at (1,-.3) {$T(x)$};
\node  at (2,-.3) {$\color{myblue} x$};
\node  at (1,3.3) {$\color{myblue} \hsigma_y(x)$};
\node  at (2.5,3.3) {$\color{myblue} \htau_x(y)$};
\node  [fill=white] at (1.2,2) {$\scriptstyle \hsigma_y(x)$};
\node  at (-.7,.8) {$T \hsigma_y(x)$};
\node  at (1.2,.8) {$y$};
\fill [myred!20, opacity=.2] (-1.5,-.6) rectangle (1.5,1.1);
\end{tikzpicture}
\caption{At the highlighted crossing, one gets $\tau_yT \hsigma_y = T$.}\label{fig:T_intertwines2}
\end{figure}
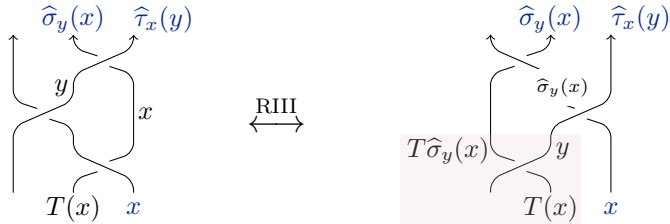
\end{proof}

\begin{rem}
All results established in this and the previous sections remain valid for infinite solutions. Indeed, the proofs in this section do not rely on the finiteness assumption. In the proof of Proposition~\ref{pro:birack_biquandle}, instead of deducing bijectivity from injectivity, one needs to check that $x \bumpeq x'$ implies $\hsigma_z(x) \bumpeq \hsigma_z(x')$ and $\sigma^{-1}_z(x) \bumpeq \sigma^{-1}_z(x')$, and similarly for $\tau$. This is done using arguments similar to those already present in the proof. We thank the reviewer for pointing out that we did not really need the finiteness, and for encouraging us to rework some of the proofs.
\end{rem}

\section{Bijective $1$-cocycles}\label{s:1cocycles}

In this section we prove a technical lemma on passing to quotients in bijective $1$-cocycles. We illustrate it with the example of the structure group $G_{(X,r)}$ of a solution $(X,r)$, and the structure group $G_{(X,\op_r)}$ of its structure rack $(X,\op_r)$. This lemma is instrumental for constructing a nice finite quotient of $G_{(X,r)}$ in Section~\ref{s:QuotientsSol}.

Let a group $G$ act on a group $H$ on the right by group automorphisms. That is, the following conditions hold for all $g_i \in G, h_i \in H$:
\begin{align*}
&h^{g_1 g_2} = (h^{g_1})^{g_2}, && h^1 = h,\\
&(h_1h_2)^{g} = h_1^{g}h_2^{g}, && 1^{g}=1.
\end{align*}

A \emph{group $1$-cocycle} is a map $\varphi \colon G \to H$ satisfying
\[\varphi(g_1 g_2) = \varphi(g_1)^{g_2}\, \varphi(g_2)\]
for all $g_i \in G$. It can be seen as a twisted group morphism. 

The definitions of monoid/semigroup actions and $1$-cocycles repeat verbatim the above definitions. Note that we consider only actions by monoid/semigroup automorphisms, i.e., the maps $h \mapsto h^g$ are all bijective.

We are particularly interested in bijective $1$-cocycles and their quotients. 

\begin{lem}\label{lem:cocycles}
Let $\varphi \colon G \to H$ be a bijective (semi-)group/monoid $1$-cocycle. Let 
\[\cR_G=\{\, g_i=g'_i \, | \, i \in I \,\}, \qquad\qquad \cR_H=\{\, \varphi(g_i)=\varphi(g'_i)  \, | \,  i \in I\,\}\] 
be a set of relations we would like to impose on~$G$, and its image in~$H$ respectively. Thus, $g_i, g'_i \in G$ for all $i$. Assume that
\begin{enumerate}
\item $h^{g_i}=h^{g'_i}$ for all $h\in H$ and $i \in I$; 
\item $\cR_H=\{\, \varphi(g_i)^g=\varphi(g'_i)^g  \, | \,  i \in I\,\}$ for all $g\in G$.
\end{enumerate}
Then the $G$-action on $H$ induces a $\raisebox{.5mm}{$G$}\big/\raisebox{-.5mm}{$\cR_G$}$-action on $\raisebox{.5mm}{$H$}\big/\raisebox{-.5mm}{$\cR_H$}$, and $\varphi$ induces a bijective (semi-)group/monoid $1$-cocycle $\overline{\varphi} \colon \raisebox{.5mm}{$G$}\big/\raisebox{-.5mm}{$\cR_G$} \to \raisebox{.5mm}{$H$}\big/\raisebox{-.5mm}{$\cR_H$}$ with respect to this induced action.
\end{lem}

\begin{proof}
Condition~1 is precisely what is needed to get a $\raisebox{.5mm}{$G$}\big/\raisebox{-.5mm}{$\cR_G$}$-action on $H$, and condition~2 allows one to descend to $\raisebox{.5mm}{$H$}\big/\raisebox{-.5mm}{$\cR_H$}$. The cocycle condition for~$\overline{\varphi}$ follows from that for~${\varphi}$. It remains to check the bijectivity. Passing to the quotient $\raisebox{.5mm}{$G$}\big/\raisebox{-.5mm}{$\cR_G$}$ means identifying $x g_i y$ with $x g'_i y$ for all $x,y \in G$, $i \in I$. The $1$-cocycle $\varphi$ sends these pairs to $\varphi(x)^{g_i y}\, \varphi(g_i)^y\, \varphi(y)$ and $\varphi(x)^{g_i y}\, \varphi(g'_i)^y\, \varphi(y)$. Since $\varphi$ is bijective and $G$ acts on $H$ by bijections, and because of condition~2, one gets precisely all pairs $z \varphi(g_j) v$ and $z \varphi(g'_j) v$ for all $z,v \in H$, $j \in I$. So, the induced cocycle~$\overline{\varphi}$ is bijective.
\end{proof}

\begin{rem}\label{rem:right_action_cocycles}
Lemma~\ref{lem:cocycles} has a symmetric version for left actions and the corresponding notion of $1$-cocycles. Left actions are denoted by ${}^g\!h$ and $g \lact h$ here.
\end{rem}

As an example, take a YBE solution $(X,r)$, and put $G = H = SG_X$, the free semigroup on $X$. This semigroup acts on itself by $x^y = \tau_y(x)$, extended to a semigroup action by semigroup morphisms in the obvious way. This action is by bijections, since all $\tau_y$ are bijective. By the freeness, there exists a unique semigroup $1$-cocycle $\varphi \colon SG_X \to SG_X$ with $\varphi(x)=x$ for all $x \in X$. It is clearly bijective. Put
\[\cR_G=\{\, xy = \sigma_x(y)\tau_y(x) \, | \, x,y \in X \,\}. \]
Then $\varphi(xy)=x^y \, y = \tau_y(x) y$, and $\varphi(\sigma_x(y)\tau_y(x))=\sigma_x(y)^{\tau_y(x)} \, \tau_y(x) = (y \op_r \tau_y(x)) \, \tau_y(x)$. Since $\tau_y$ is bijective,
\[\cR_H=\{\, xy = (y \op_r x) \, x \, | \, x,y \in X \,\}. \]
For these data, condition~1 from Lemma~\ref{lem:cocycles} follows from~\eqref{eqn:tau_action}, and condition~2 from the  following property, established in \cite[Theorem 2.3.(i)]{Sol}. Observe that our proof does not exploit the finiteness of~$X$.
\begin{lem}\label{lem:TauActsByRackMor}
Let $(X,r)$ be a solution. For all $z \in X$, both $\tau_z$, and $\htau_z$ are rack automorphisms of $(X,\op_r)$: for all $x,y \in X$, one has 
\begin{align*}
&\tau_z(x \op_r y)=\tau_z(x) \op_r \tau_z(y), && \htau_z(x \op_r y)=\htau_z(x) \op_r \htau_z(y).
\end{align*}
\end{lem}

\begin{proof}
For $\tau_z$, the argument is similar to that from Lemma~\ref{lem:SqLeftRight}. In Fig.~\ref{fig:TauActsByRackMor}, we give its graphical version. The proof for $\htau_z$ in analogous. Graphically, it suffices to turn the diagrams from Fig.~\ref{fig:TauActsByRackMor} upside down.
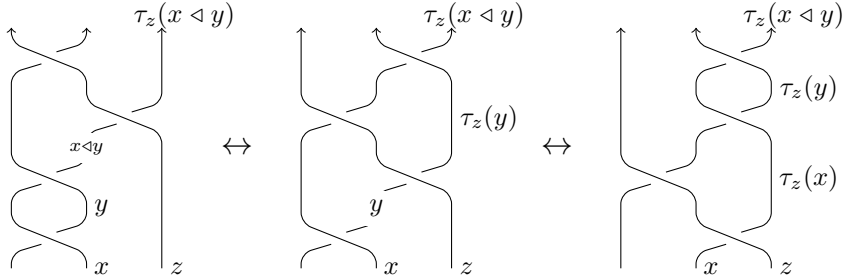
\begin{figure}[h]\centering
\begin{tikzpicture}[xscale=1,yscale=.8]
\draw [rounded corners](1,-1)--(1,-.75)--(0,-.25)--(0,0.25)--(0.4,0.4);
\draw [rounded corners](0.6,0.6)--(1,0.75)--(1,1.25)--(1.4,1.4);
\draw [rounded corners,->](1.6,1.6)--(2,1.75)--(2,3);
\draw [rounded corners](0,-1)--(0,-.75)--(0.4,-.6);
\draw [rounded corners](0.6,-.4)--(1,-.25)--(1,0.25)--(0,0.75)--(0,2.25)--(0.4,2.4);
\draw [rounded corners,->](0.6,2.6)--(1,2.75)--(1,3);
\draw [rounded corners,->](2,-1)--(2,1.25)--(1,1.75)--(1,2.25)--(0,2.75)--(0,3);
\node  at (1.2,-1) {$x$};
\node  at (2.2,-1) {$z$};
\node  at (1.2,0) {$y$};
\node  [fill=white] at (1,1) {$\scriptstyle x \op y$};
\node  at (2.3,3.2) {$\tau_z(x \op y)$};
\node  at (3.2,1){\Large $\leftrightarrow$\hspace*{.4cm}};
\end{tikzpicture}
\begin{tikzpicture}[xscale=1,yscale=.8]
\draw [rounded corners](0,0)--(0,0.25)--(0.4,0.4);
\draw [rounded corners](0.6,0.6)--(1,0.75)--(1,1.25)--(1.4,1.4);
\draw [rounded corners,->](1.6,1.6)--(2,1.75)--(2,3.25)--(1,3.75)--(1,4);
\draw [rounded corners](1,0)--(1,0.25)--(0,0.75)--(0,2.25)--(0.4,2.4);
\draw [rounded corners](0.6,2.6)--(1,2.75)--(1,3.25)--(1.4,3.4);
\draw [rounded corners,->](1.6,3.6)--(2,3.75)--(2,4);
\draw [rounded corners,->](2,0)--(2,1.25)--(1,1.75)--(1,2.25)--(0,2.75)--(0,4);
\node  at (3.6,2){\Large $\leftrightarrow$\hspace*{.4cm}};
\node  at (1.2,0) {$x$};
\node  at (2.2,0) {$z$};
\node  [fill=white] at (1,1) {$y$};
\node  at (2.5,2.5) {$\tau_z(y)$};
\node  at (2.3,4.2) {$\tau_z(x \op y)$};
\end{tikzpicture}
\begin{tikzpicture}[xscale=1,yscale=.8]
\draw [rounded corners](1,1)--(1,1.25)--(1.4,1.4);
\draw [rounded corners](1.6,1.6)--(2,1.75)--(2,3.25)--(1,3.75)--(1,4.25)--(1.4,4.4);
\draw [rounded corners,->](1.6,4.6)--(2,4.75)--(2,5);
\draw [rounded corners](0,1)--(0,2.25)--(0.4,2.4);
\draw [rounded corners](0.6,2.6)--(1,2.75)--(1,3.25)--(1.4,3.4);
\draw [rounded corners,->](1.6,3.6)--(2,3.75)--(2,4.25)--(1,4.75)--(1,5);
\draw [rounded corners,->](2,1)--(2,1.25)--(1,1.75)--(1,2.25)--(0,2.75)--(0,5);
\node  at (1.2,1) {$x$};
\node  at (2.2,1) {$z$};
\node  at (2.5,2.5) {$\tau_z(x)$};
\node  at (2.5,4) {$\tau_z(y)$};
\node  at (2.3,5.2) {$\tau_z(x \op y)$};
\end{tikzpicture}
\caption{The three colors $x,y,z$ in the left diagram uniquely determine all colors in all the diagrams. From the rightmost diagram, one gets $\tau_z(x \op_r y)=\tau_z(x) \op_r \tau_z(y)$.}\label{fig:TauActsByRackMor}
\end{figure}
\end{proof}

Our arguments work for monoids instead of semigroups; in this case one necessarily has $\varphi(1)=1$. Lemma~\ref{lem:cocycles} then yields bijective semigroup/monoid $1$-cocycles 
\[\varphi_{SG} \colon SG_{(X,r)} \to SG'_{(X,\op_r)}, \qquad\qquad \varphi_{Mon} \colon Mon_{(X,r)} \to Mon'_{(X,\op_r)},\] 
where \emph{structure semigroups/monoids} of solutions or racks are defined by the same generators and relations as structure groups. 

Next, let $(X,r)$ be a biquandle, and consider the \emph{double} $D(X)$ of the set~$X$, consisting of the elements $x=x^{+1}$ and $x^{-1}$, $x \in X$. This time, choose $G = H = Mon_{D(X)}$, the free monoid on $D(X)$. This monoid acts on itself by
\begin{align}\label{eqn:ActionStrGroups}
x^{\varepsilon_x} \ract y^{\varepsilon_y} = \tau_y^{\varepsilon_y} (x)^{\varepsilon_x}, \qquad\qquad x,y \in X,\: \varepsilon_x,\varepsilon_y \in \{\pm 1\},
\end{align}
extended to a monoid action by monoid morphisms in the obvious way. Together with the personalized notation $\ract$, we use the subscript notation $x^y$ for this or any other action, when it does not create ambiguity. This action is again by bijections. By the freeness, there exists a unique monoid $1$-cocycle $\varphi \colon Mon_{D(X)} \to Mon_{D(X)}$ with $\varphi(1)=1$, $\varphi(x)=x$, and $\varphi(x^{-1})=t(x)^{-1}$ for all $x \in X$; see~\eqref{eqn:biquandle} for the map~$t$. This cocycle is bijective, since so is the map~$t$. Put
\[\cR_G=\{\, xy = \sigma_x(y)\tau_y(x) \, | \, x,y \in X \,\} \,\sqcup\, \{\, xx^{-1} = x^{-1}x=1 \, | \, x \in X \,\}. \]
Then $\varphi(xx^{-1})=x^{x^{-1}} t(x)^{-1} = \tau_x^{-1}(x)t(x)^{-1}=t(x) t(x)^{-1}$, since $\tau_x(t(x))=x$ due to~\eqref{eqn:biquandle}. Similarly, $\varphi(x^{-1}x)=(t(x)^{-1})^x x =\tau_x(t(x))^{-1} x = x^{-1} x$. Since $t$ is bijective, one concludes
\[\cR_H=\{\, xy = (y \op_r x) \, x \, | \, x,y \in X \,\}  \,\sqcup\, \{\, xx^{-1} = x^{-1}x=1 \, | \, x \in X \,\}. \]
For relations of the first type (braid type), conditions from Lemma~\ref{lem:cocycles} hold for the same reasons as before. For relations of the second type (group type), they follow from definition~\eqref{eqn:ActionStrGroups}. Lemma~\ref{lem:cocycles} then yields a $1$-cocycles for quotients:

\begin{thm}\label{thm:cocycle}
Let $(X,r)$ be a solution. There exists a unique bijective group $1$-cocycle $J \colon G_{(X,r)} \to G'_{(X,\op_r)}$ such that the following diagram commutes:
\[\xymatrix@!0 @R=1cm @C=1.5cm{&X \ar[dl]_{\iota} \ar[dr]^{\iota} &\\
G_{(X,r)} \ar[rr]^{J} && G'_{(X,\op_r)}.}\]
Here 
\begin{itemize}
\item $G_{(X,r)}$ and $G'_{(X,\op_r)}$ are the structure groups of $(X,r)$ and of its structure rack $(X,\op'_r)$ respectively;
\item $G_{(X,r)}$ acts on $G'_{(X,\op_r)}$ on the right by group automorphisms extending~\eqref{eqn:ActionStrGroups};
\item the two maps $\iota$ are defined by $\iota(x)=x$ for all $x \in X$.
\end{itemize}
\end{thm}

\begin{proof}
If $(X,r)$ is a biquandle, then the argument preceding the theorem yields a bijective group $1$-cocycle $J=\overline{\varphi} \colon G_{(X,r)} \to G'_{(X,\op_r)}$, with $\overline{\varphi}(x)=x$ for all $x \in X$ by construction. This cocycle is unique, since a group $1$-cocycle is uniquely determined by its values on the generators. If $(X,r)$ is a general solution, one can replace it with its induced biquandle using Propositions~\ref{pro:commdiag_Sol_Q_BQ} and~\ref{pro:commdiag_R_Q_BQ}. Alternatively, one can work directly with non-biquandle solutions, replacing the map $t$ with the map $T$ from Lemma~\ref{lem:SqProperties}.
\end{proof}

This result appears in a weaker form and/or with a much more involved proof in \cite{Sol,LYZ,LebVen}. The map $J$ is called the \emph{$J$ map}, or the \emph{guitar map} there. 

For involutive $r$, $G'_{(X,\op_r)}$ is the free abelian group on $X$, and the diagram from the theorem yields the injectivity of $r$, its powerful characterization in group-theoretic terms \cite{ESS}, and an $I$-structure on $G_{(X,r)}$ \cite{GIVdB}. In the general case, the situation is much more delicate.

We finish this section with a graphical way of computing $\varphi(x_1\ldots x_m)=x'_1\ldots x'_m$ for $x_i \in X$, presented in Fig.~\ref{fig:Guitar}. The $1$-cocycle condition for $\varphi$ yields $x'_i=x_i^{x_{i-1}\ldots x_1} = \tau_{x_1}\cdots\tau_{x_{i-1}}(x_i)$. 
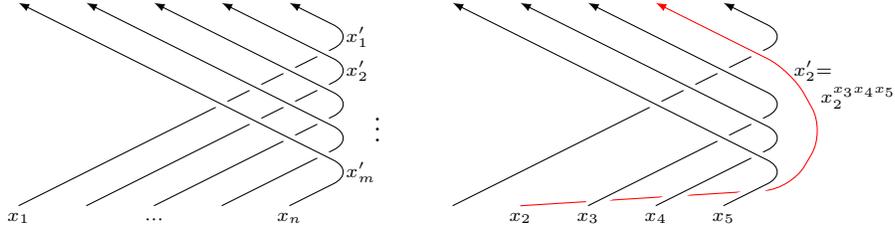
\begin{figure}[h]
\centering
\begin{tikzpicture}[scale=0.9,>=latex]
 \draw [->, rounded corners=10] (0,0) -- (5,2.5) -- (4,3);
 \draw [line width=4pt,white, rounded corners=10] (1,0) -- (5,2) -- (3,3); 
 \draw [->, rounded corners=10] (1,0) -- (5,2) -- (3,3); 
 \draw [line width=4pt,white, rounded corners=10] (2,0) -- (5,1.5) -- (2,3); 
 \draw [->, rounded corners=10] (2,0) -- (5,1.5) -- (2,3);
 \draw [line width=4pt,white, rounded corners=10] (3,0) -- (5,1) -- (1,3); 
 \draw [->, rounded corners=10] (3,0) -- (5,1) -- (1,3); 
 \draw [line width=4pt,white, rounded corners=10] (4,0) -- (5,0.5) -- (0,3);   
 \draw [->, rounded corners=10] (4,0) -- (5,0.5) -- (0,3);  
 \node at (0,-0.4) [above] {$\scriptstyle{x_1}$};  
 \node at (2,-0.4) [above] {$\scriptstyle{\cdots}$}; 
 \node at (4,-0.4) [above] {$\scriptstyle{x_n}$}; 
 \node at (4.7,2.5) [right] {$\scriptstyle{x'_1}$};   
 \node at (4.7,2) [right] {$\scriptstyle{x'_2}$}; 
 \node at (5,1.25) [right] {$\;\scriptstyle{\vdots}$};  
 \node at (4.7,0.5) [right] {$\scriptstyle{x'_m}$};  
\end{tikzpicture} 
\hspace*{15pt}
\begin{tikzpicture}[scale=0.9,>=latex]
 \draw [->, rounded corners=10] (0,0) -- (5,2.5) -- (4,3);
 \draw [line width=4pt,white, rounded corners=10] (1,0) -- (5,0.25)  -- (5.5,1.1) -- (5,2) -- (3,3); 
 \draw [->, rounded corners=10,red] (1,0) -- (5,0.25)  -- (5.5,1.1) -- (5,2) -- (3,3); 
 \draw [line width=4pt,white, rounded corners=10] (2,0) -- (5,1.5) -- (2,3); 
 \draw [->, rounded corners=10] (2,0) -- (5,1.5) -- (2,3);
 \draw [line width=4pt,white, rounded corners=10] (3,0) -- (5,1) -- (1,3); 
 \draw [->, rounded corners=10] (3,0) -- (5,1) -- (1,3); 
 \draw [line width=4pt,white, rounded corners=10] (4,0) -- (5,0.5) -- (0,3);   
 \draw [->, rounded corners=10] (4,0) -- (5,0.5) -- (0,3);  
 \node at (1,-0.4) [above] {$\scriptstyle{x_2}$};
 \node at (2,-0.4) [above] {$\scriptstyle{x_3}$};  
 \node at (3,-0.4) [above] {$\scriptstyle{x_4}$}; 
 \node at (4,-0.4) [above] {$\scriptstyle{x_5}$}; 
 \node at (4.9,2) [right] {$\scriptstyle{x'_2=}$}; 
 \node at (5.3,1.6) [right] {$\scriptstyle{x_2^{x_3 x_4 x_5}}$};    
\end{tikzpicture} 
   \caption{The guitar map $J$.}\label{fig:Guitar}
\end{figure}

\begin{rem}\label{r:braces}
Bijective group $1$-cocycles have many alternative descriptions: as braided commutative groups, as skew braces, as a skew version of linear cycle sets, as certain factorized groups \cite{LYZ, GIMajid, GV, SmoVen}. The results of this section can be translated into any of these languages. 
\end{rem}

\section{Finite quotients for structure groups of racks}\label{s:QuotientsRack}

We now describe finite quotients of structure groups of racks which preserve injectivity (that is, the injectivity of the map $\iota \colon X \to G_{(X,\op)}$).

\begin{notation}
Let $(X,\op)$ be a finite rack. Denote by 
\[K_{\op}=\#\Orb(X,\op)\] 
the number of its {orbits} with respect to the actions $\rho_y \colon  x \mapsto x \op y$. Further, for $y \in X$, denote by $D_y$ the minimal integer $D$ satisfying 
\[D \geq 2 \qquad \text{ and } \qquad \rho_y^D=\Id_X.\]
\end{notation}

\begin{thm}\label{thm:quotient_rack}
Let $(X,\op)$ be a finite rack. The powers $x^{D_x}$, $x \in X$, generate a central subgroup $Z_{(X,\op)}$ of the structure group $G_{(X,\op)}$. This subgroup is free abelian of rank $K_{\op}$. The quotient \[\oG_{(X,\op)} = G_{(X,\op)} / Z_{(X,\op)}\] is finite. Finally, the map $\iota \colon X \to G_{(X,\op)}$ is injective if and only if it remains so when composed with the quotient map $G_{(X,\op)} \twoheadrightarrow \oG_{(X,\op)}$.
\end{thm}

The last property from the theorem is called the \emph{injectivity preservation} for $\oG$. 

In the $1$-orbit case, these finite quotients were described in~\cite{GraHeckVen}. 

\begin{proof}
Since $\rho_{x \op y} = \rho_y\rho_x\rho_y^{-1}$ for all $x,y \in X$, one has $D_x=D_z$ for all $x,z$ from the same orbit of $X$. Further, relations
\begin{align*}
x y^n &= y^n \rho_y^n(x),&  x^n y &= y (x \op y)^n
\end{align*}
in $G_{(X,\op)}$ imply that the powers $y^{D_y}$ are central in $G_{(X,\op)}$, and that 
\[x^{D_x} = (x \op y)^{D_x} = (x \op y)^{D_{x \op y}} \;\text{ for all } x,y \in X.\] 
Thus, $x^{D_x} = z^{D_z}$ for all $x,z$ from the same orbit. Let $\O_1, \ldots, \O_k$, where $k=K_{\op}$, be the orbits of $(X,\op)$, let $x_1,\ldots,x_k$ be their representatives, and put $D_i=D_{x_i}$. From the defining relations of $G_{(X,\op)}$, one sees that there exists a group
morphism
\begin{align*}
\theta \colon G_{(X,\op)} &\to \oplus_i \Z \O_i \cong \Z^k,\\
x &\mapsto \O(x),
\end{align*} 
where $\O(x)$ is the orbit of~$x$. This map is clearly well defined and surjective. The powers $x_i^{D_{i}}$ are sent to linearly independent elements $\O_i^{D_i}$ of $\oplus_i \Z \O_i$. Therefore the abelian group $Z_{(X,\op)}$ is free abelian of rank $k$, with generators $x_1^{D_{1}},\ldots,x_k^{D_{k}}$.

To prove the finiteness of the quotient $\oG_{(X,\op)}$, we will show that any of its elements can be represented by a word from $G_{(X,\op)}$ with at most $2n(d-1)$ letters $x^{\pm 1}$, $x \in X$, where $n=\# X$, and $d$ is the maximum of the $D_x$. Indeed, a word with more than $2n(d-1)$ letters contains one of the $2n$ letters $x^{\pm 1}$ at least $d$ times. The relation $yx=x(y \op x)$ allows one to pull all occurrences of this letter to the left; the remaining letters might change, but not their number. Since $d \geq D_x$, one can replace $(x^{\pm 1})^d$ with $(x^{\pm 1})^{d-D_x}$ to get a shorter representative of the same word.

Finally, assume that the map $\iota \colon X \to G_{(X,\op)}$ is injective, but its composition with $G_{(X,\op)} \to \oG_{(X,\op)}$ is not. It means that there exist $x \neq z$ in $X$ with 
\[xz^{-1} = (x_1^{D_{1}})^{\alpha_1}\cdots (x_k^{D_{k}})^{\alpha_k} \in G_{(X,\op)}\]
for some $\alpha_i \in \Z$. Applying the above map $\theta$ to both sides, one gets $\O(x)-\O(z)=\sum_i D_i\alpha_i \O_i$. Since the $\O_i$ form a basis of $\oplus_i \Z \O_i$, and since $D_i \geq 2$ for all $i$, this implies $D_i\alpha_i=0$ and thus $\alpha_i=0$ for all $i$. But then $xz^{-1} = 1$ in $G_{(X,\op)}$, hence $\iota(x)=\iota(z)$, which is impossible for injective~$\iota$.
\end{proof}

Imposing $D_x \geq 1$ instead of $D_x \geq 2$ would be more natural, and would make most of the assertions of the theorem still hold true. The only exception is the injectivity preservation. Indeed, for the trivial rack $x \op y = x$, $G_{(X,\op)}$ is the free abelian group on $X$, so $\iota \colon X \to G_{(X,\op)}$ is injective. Further, $\rho_x^1=\Id_X$ for all $x$, and the powers $x^1=x$ generate the whole $G_{(X,\op)}$, so the quotient $\oG_{(X,r)}$ would be zero with this definition.

\begin{rem}\label{rem:quotient_rack_sym}
One can replace the groups $G_{(X,\op)}$ and $\oG_{(X,\op)}$ in Theorem~\ref{thm:quotient_rack} with their symmetric versions $G'_{(X,\op)}$ and $\oG'_{(X,\op)}$, constructed using the solution $r'_{\op}(x,y)=(y \op x, x)$ instead of $r_{\op}$. To see this, use the map $x_1^{\varepsilon_1} \ldots x_s^{\varepsilon_s} \mapsto x_s^{\varepsilon_s} \ldots x_1^{\varepsilon_1}$.
\end{rem}

Observe that the map~$\iota$ and its composition $\overline{\iota}$ with $G_{(X,\op)} \twoheadrightarrow \oG_{(X,\op)}$ are rack morphisms, where both groups are seen as \emph{conjugation quandles}, with $g \op_{Conj} h = h^{-1}gh$. Theorem~\ref{thm:quotient_rack} then yields a characterization of injective racks:

\begin{cor}
For a finite rack $(X,\op)$, the following conditions are equivalent:
\begin{enumerate}
\item $(X,\op)$ is injective;
\item $(X,\op)$ is isomorphic to a sub-quandle of a finite conjugation quandle;
\item $(X,\op)$ is isomorphic to a sub-quandle of a conjugation quandle.
\end{enumerate}
\end{cor}

\begin{proof}
Theorem~\ref{thm:quotient_rack} yields $1 \Rightarrow 2$, and $2 \Rightarrow 3$ is trivial. Recall the universal property of the structure group $G_{(X,\op)}$: for any group $H$ and any rack map $\phi \colon (X,\op) \to (H,\op_{Conj})$, there is a unique group map ${\phi^*} \colon G_{(X,\op)} \to (H,\op_{Conj})$ satisfying $\phi = {\phi^*} \iota$:
\[\xymatrix@!0 @R=1cm @C=2cm{
X \ar[r]^-{\iota} \ar[rd]_{\phi} & G_{(X,\op)} \ar@{.>}[d]^{\exists ! {\phi^*}}\\
&H}\]
Now, if $\phi$ is injective, then so is $\iota$. Hence $3 \Rightarrow 1$.
\end{proof}

The injectivity question for a rack is not at all obvious, as can be seen in the examples below. In the second one, the finite quotient was crucial to establish injectivity.

\begin{exa}
Let $X=\{1,2,\dots,8\}$ be the quandle given by
\begin{align*}
    &\rho_1 =(376)(485), && \rho_2=(376)(485), && \rho_3=(168)(257), && \rho_4 =(168)(257), \\
    &\rho_5=(174)(283), && \rho_6=(174)(283), && \rho_7=(135)(246), &&\rho_8=(135)(246).
\end{align*}
In $G'_{(X,\op)}$, one has $x_1x_3=x_7x_1=x_3x_7=x_2x_3$, hence $x_1=x_2$. Thus $X$ is not injective. 
\end{exa}

\begin{exa}
Let $X=\{1,2,3,4,5,6,7,8,9,a,b,c\}$ be the quandle given by
\begin{align*}
  &\rho_1 = (34)(59)(6a)(7c)(8b), && \rho_2=(34)(59)(6a)(7c)(8b), \\ &\rho_3 = (12)(5b)(6c)(7a)(89), && \rho_4=(12)(5b)(6c)(7a)(89), \\ &\rho_5 = (19)(2a)(3c)(4b)(78), && \rho_6=(19)(2a)(3c)(4b)(78), \\  &\rho_7 = (1c)(2b)(39)(4a)(56), && \rho_8=(1c)(2b)(39)(4a)(56), \\ &\rho_9 = (15)(26)(38)(47)(bc), && \rho_a=(15)(26)(38)(47)(bc),\\ &\rho_b = (17)(28)(36)(45)(9a), && \rho_c=(17)(28)(36)(45)(9a).
\end{align*}
Computer calculations show that $\oG'_{(X,\op)} \simeq \GL(2,3)$, and $X$ embeds into this finite group. More precisely, $(X,\op)$ is isomorphic to the sub-quandle of $\GL(2,3)$ consisting of the conjugates of $\begin{pmatrix} 1 & 0\\1 & 1\end{pmatrix}$. This rack is retractable (see Definition~\ref{def:retract}).
\end{exa}

\section{Finite quotients for structure groups of YBE solutions}\label{s:QuotientsSol}

We next turn to finite quotients of the structure group of a solution $(X,r)$.

\begin{notation}
Let $(X,r)$ be a YBE solution. Denote by 
\[K_{r}=K_{\op_r}=\#\Orb(X,\op_r)\] 
the number of orbits of its structure rack. Further, for $y \in X$ and $d \in \N$, put 
\[y^{[d]}=J^{-1}(y^d) \in G_{(X,r)}.\] 
Here $J$ is the bijective group $1$-cocycle $J \colon G_{(X,r)} \to G'_{(X,\op_r)}$ from Theorem~\ref{thm:cocycle}, with respect to the action $\ract$ of $G_{(X,r)}$ on $G'_{(X,\op_r)}$ defined by~\eqref{eqn:ActionStrGroups}.
\end{notation}

\begin{lem}\label{lem:powers_formula}
For any $y \in X$ and $d \in \N$, one has
\begin{align}\label{eqn:TwistedPowers}
&y^{[d]}= T^{d-1}(y) \ldots T(y) y,
\end{align}
where the map $T \colon X \to X$ sends $x$ to $\tau_x^{-1} (x)$.
\end{lem}

Formula~\eqref{eqn:TwistedPowers} explains the name \emph{$T$-twisted powers of $y$} we will give to the $y^{[d]}$.

\begin{proof}
By definition of the map $T$, one has $T(x)^x=\tau_xT(x)=x$ for all~$x$. So, for all $m \in \N$,
\[T^m(y)^{T^{m-1}(y) \ldots T(y) y} = (T^m(y)^{T^{m-1}(y)})^{T^{m-2}(y) \ldots T(y) y} = T^{m-1}(y)^{T^{m-2}(y) \ldots T(y) y},\]
which, after several iterations, yields $y$. Using this and the $1$-cocycle property for $J$, one gets $J(T^{d-1}(y) \ldots T(y) y)=y^d$, so $y^{[d]}= J^{-1}(y^d) = T^{d-1}(y) \ldots T(y) y$.
\end{proof}

\begin{notation}
By \eqref{eqn:tau_action}--\eqref{eqn:sigma_action}, the group $G_{(X,r)}$ acts on the set~$X$ in several ways. We will denote these actions as follows:
\begin{align*}
&x \racts y = x^y = \tau_y(x), && y \lacts x = {}^y\!x = \sigma_y(x),\\
&x \ractsi y = \htau_y(x), && y \lactsi x =  \hsigma_y(x)
\end{align*}
for all $x,y \in X$. 
The superscript notation will be used only when it cannot be mistaken for the analogous notation for the $G_{(X,r)}$-action on $G'_{(X,\op_r)}$.
\end{notation}

\begin{lem}\label{lem:degree_exists}
Let $(X,r)$ be a solution. For any $y \in X$ there exists a $d \in \N$ such that $x \racts y^{[m]} = y^{[m]} \lacts x = x$ for all $x \in X$ and all $m \in \N$ divisible by $d$.
\end{lem}

\begin{proof}
By Lemma~\ref{lem:SqProperties}, $T$ is a bijection. Since $X$ is finite, $T^p=\Id_X$ for some $p \in \N$. The map $\psi(x)= x \racts y^{[p]}$ is a bijection on~$X$ as well, so $\psi^q=\Id_X$ for some $q \in \N$. From~\eqref{eqn:TwistedPowers}, one deduces $y^{[pm]} = (y^{[p]})^m$ for all $m \in \N$, hence $x \racts y^{[pqs]} = \psi^{qs}(x)=x$ for all $x \in X$, $s \in \N$. A similar argument yields a $q' \in \N$ with $y^{[pq's]} \lacts x = x$ for all $x \in X$, $s \in \N$. Then $d=p (q \vee q')$, where $\vee$ stands for the least common multiple, satisfies our requirements. 
\end{proof} 

Lemma~\ref{lem:degree_exists} justifies the following definition.

\begin{defn}\label{def:degree}
Let $(X,r)$ be a solution, and $y \in X$. The \emph{degree} $d_y$ of $y$ is the minimal positive integer~$d$ satisfying the following conditions:
\begin{enumerate}
\item\label{it:degreeNonTrivial} $d$ is even if $\rho_y=\Id_X$ for the map $\rho_y \colon x \mapsto x \op_r y$;
\item\label{it:degreeSD} $\rho_y^d=\Id_X$;
\item\label{it:degreeActions} $x \racts y^{[d]} = y^{[d]} \lacts x = x$ for all $x \in X$.
\end{enumerate} 
\end{defn}

Everything is now ready for the main result of this section:

\begin{thm}\label{thm:quotient}
Let $(X,r)$ be a solution. The powers $x^{[d_x]}$, $x \in X$, generate a normal subgroup $Z_{(X,r)}$ of the structure group $G_{(X,r)}$. This subgroup is free abelian of rank $K_{r}$. The quotient \[\oG_{(X,r)} = G_{(X,r)} / Z_{(X,r)}\] is finite. Finally, the map $\iota \colon X \to G_{(X,r)}$ is injective if and only if it remains so when composed with the quotient map $G_{(X,r)} \to \oG_{(X,r)}$.
\end{thm}

The theorem gives us the following short exact sequence:
\begin{align*}
&0 \to \Z^{K_{r}} \to G_{(X,r)} \to \oG_{(X,r)} \to 0.
\end{align*}

The proof combines finite quotients of the structure group $G'_{(X,\op_r)}$ of the structure rack of our solution (Theorem~\ref{thm:quotient_rack} and Remark~\ref{rem:quotient_rack_sym}); the bijective cocycle $J$ as a means of transport between $G_{(X,r)}$ and $G'_{(X,\op_r)}$ (Theorem~\ref{thm:cocycle} and Lemma~\ref{lem:cocycles}); and a study of $T$-twisted powers, on which we now concentrate.

Our solution $r$ extends from $X$ to the (infinite) set $T(X)=\sqcup_{n\geq 0}X^{\times n}$ in the obvious way, as shown in Fig.~\ref{fig:TX}. We keep the notation $r$ for this extended solution, and $\sigma$ and $\tau$ for its components.
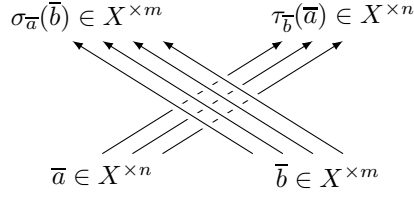
\begin{figure}[h]
\begin{tikzpicture}[xscale=0.4,yscale=0.25,>=latex]
\draw [->] (1,0)--(7,6);
\draw [->] (2,0)--(8,6);
\draw [->] (3,0)--(9,6);
\draw [line width=4pt,white] (6,0)--(0,6);
\draw [line width=4pt,white] (7,0)--(1,6);
\draw [line width=4pt,white] (8,0)--(2,6);
\draw [line width=4pt,white] (9,0)--(3,6);
\draw [->] (6,0)--(0,6);
\draw [->] (7,0)--(1,6);
\draw [->] (8,0)--(2,6);
\draw [->] (9,0)--(3,6);
\node [below] at (1,0) {${\overline{a}} \in X^{\times n}$};
\node [below] at (8.5,0) {${\overline{b}} \in X^{\times m}$};
\node [above] at (0.5,6) {$\sigma_{\overline{a}}(\overline{b}) \in X^{\times m}$};
\node [above] at (9,6) {$\tau_{\overline{b}}(\overline{a}) \in X^{\times n}$};
\end{tikzpicture}
\caption{Solution on~$X$ extended to~$T(X)$.}\label{fig:TX}
\end{figure}

\begin{lem}\label{lem:T_r}
For any $y,z \in X$ and $m\in \N$, one has
\begin{align*}
&r(y^{[m]},z) = (y^{[m]} \lacts z,(y^z)^{[m]}), && r(z,y^{[m]}) = (u^{[m]},z \racts y^{[m]}),
\end{align*}
where $u=\htau_{z'}^{-1}(y)$, $z' = z \racts y^{[m]}$, and $\htau$ is the right component of~$r^{-1}$. 
\end{lem}

The first relations implies that the map $\Delta_m \colon X \to X^{\times m}$, \, $y \mapsto y^{[m]}$, is $G_{(X,r)}$-equivariant, where $G_{(X,r)}$ acts on the powers of~$X$ on the right by the (extension of the) $\tau$ component of~$r$: $\tau_z(y^{[m]})=\tau_z(y)^{[m]}$. The equivariance with respect to the left actions via~$\sigma$ does not hold in general.

\begin{proof}
Both statements follow by repeatedly applying the following property:
\begin{align*}
&r(T(y)y,z) = ((T(y)y) \lacts z,T(y^z)y^z), \\
& r(z,yT^{-1}(y)) = ({}^z\!yT^{-1}({}^z\!y),z \racts (yT^{-1}(y))).
\end{align*}
We'll prove the first property, the second one being similar. The extension of~$r$ works as follows:
\[r(T(y)y,z) = (\Id_X \times c)r_1r_2(T(y),y,z)=((T(y)y) \lacts z,v y^z)\]
for a certain $v \in X$. Here $c$ is the concatenation map $X \times X \to X^{\times 2}$. It remains to show $v=T(y^z)$, that is, $v^{(y^z)}=y^z$. This is done as follows:
\[((T(y)y) \lacts z,{}^v\!(y^z),v^{(y^z)})=r_2r_1r_2(T(y),y,z)\]
\[=r_1r_2r_1(T(y),y,z)=r_1r_2(\bullet,y,z)=r_1(\bullet,{}^y\!z,y^z)=(\bullet,\bullet,y^z),\]
where the $\bullet$ replace irrelevant entries.
\end{proof}

\begin{rem}\label{rem:Cabling}
Lemma~\ref{lem:T_r} implies that the map 
\[r^{[m]}:(x,y) \mapsto (T^{-(m-1)}(x^{[m]} \lacts T^{m-1}(y)) ,x \racts y^{[m]})\]
defines a solution on~$X$.  Even better: it yields the \emph{cabling functor} 
\[\operatorname{Cab}_m \colon \SolCat \to \SolCat.\] 
The name comes from its diagrammatic interpretation. By Lemma~\ref{lem:degree_exists}, the sequence of functors $(\operatorname{Cab}_m)_{m \geq 1}$ is periodic.
\end{rem}

\begin{defn}\label{def:OrbitsSolution}
By \emph{orbits} of a solution $(X,r)$ we will mean orbits with respect to the actions $y \mapsto \tau_z(y)$ and $y \mapsto \sigma_z(y)$ for all $z \in X$. 
\end{defn}

This notion should be distinguished from the finer notion of orbits of the structure rack $(X,\op_r)$ of our solution. Thus, for involutive $r$, the structure rack is trivial, so all its orbits contain one element only, whereas the orbits of the solution itself can be much bigger.

\begin{lem}\label{lem:degrees_coinside}
The degrees of all elements from the same orbit of $(X,r)$ coincide.
\end{lem}

\begin{proof}
It suffices to check that conditions \eqref{it:degreeNonTrivial}--\eqref{it:degreeActions} from Definition~\ref{def:degree} do not change when $y$ is replaced with $y \racts z = \tau_z(y)$ or with $\htau^{-1}_z(y)$ for any $z$. Indeed, $r(z,y)=(\sigma_z(y),\tau_y(z))$ implies $\htau_{\tau_y(z)}\sigma_z(y)=y$, so in the definition of the orbits of $(X,r)$, the actions $\sigma_z$ can be replaced with $\htau^{-1}_z$.

For \eqref{it:degreeNonTrivial}--\eqref{it:degreeSD}, the desired equivalence follows from Lemma~\ref{lem:TauActsByRackMor}. Condition~\eqref{it:degreeActions} needs more work. Take any $y,z \in X$ and an $m\in \N$. Suppose that $x \racts y^{[m]} = y^{[m]} \lacts x = x$ holds for all $x \in X$. It suffices to check the relations
\begin{align*}
&x \racts \tau_z(y)^{[m]} = \tau_z(y)^{[m]} \lacts x = x \racts \htau^{-1}_z(y)^{[m]} = \htau^{-1}_z(y)^{[m]} \lacts x = x.
\end{align*}
Checking the triviality of the actions of $\tau_z^{-1}(y)^{[m]}$ and $\htau_z(y)^{[m]}$ is not necessary, since $\tau_z^{-1}=\tau_z^{p}$ and $\htau_z=(\htau^{-1}_z)^q$ for some $p,q \in \N$ due to the finiteness of~$X$. 

Lemma~\ref{lem:T_r} yields the relation $y^{[m]} z = z \tau_z(y)^{[m]}$ in $G_{(X,r)}$. By assumption, $y^{[m]}$ acts trivially on~$X$ with respect to the two actions $\lacts$ and $\racts$. Thus $\tau_z(y)^{[m]}$ acts trivially as well. Similarly, Lemma~\ref{lem:T_r} yields the relation $z y^{[m]} = \htau^{-1}_z(y)^{[m]} z$ in $G_{(X,r)}$. The triviality of the actions of $\htau^{-1}_z(y)^{[m]}$ follows.
\end{proof}

\begin{proof}[Proof of Theorem~\ref{thm:quotient}]
By Lemmas \ref{lem:T_r} and \ref{lem:degrees_coinside}, the relation
\[x^{[d_x]} \, z = z \, \tau_z(x)^{[d_{x}]} = z \,\tau_z(x)^{[d_{\tau_z(x)}]}\]
holds in $G_{(X,r)}$ for all $x,z \in X$. Therefore, the subgroup $Z_{(X,r)}$ generated by the $x^{[d_x]}$ is normal. 

Take a $y \in X$. Since $x\racts y^{[d_y]} = x$ for all $x \in X$, one gets
\begin{align}\label{eqn:PowersActTrivially}
&g\ract y^{[d_y]} = g && \text{ for all }\; g \in G'_{(X,\op_r)}.
\end{align}
Therefore, 
\[J(x^{[d_x]}\, y^{[d_y]})= (J(x^{[d_x]}) \ract y^{[d_y]}) \, J(y^{[d_y]}) = J(x^{[d_x]})J(y^{[d_y]}).\]
A similar argument yields $J((x^{[d_x]})^{-1}) = J(x^{[d_x]})^{-1}$. So, $J$ restricted to $Z_{(X,r)}$ is a group isomorphism. The degree $d_x$ of~$x$ is a multiple of~$D_x$ by definition\footnote{We imposed $d_x$ to be even when $\rho_x=\Id_X$ precisely to get this property.}; recall that $D_x$ is the minimal integer $D$ satisfying $D \geq 2$ and $\rho_y^D=\Id_X$. Also, $J(x^{[d_x]})=x^{d_x} \in G'_{(X,\op_r)}$. Thus the image $J(Z_{(X,r)})$ is the subgroup of $Z_{(X,\op_r)}\cong \bigoplus_{i=1}^{K_{r}} x_i^{D_i}$ generated by the $x_i^{\alpha_i D_i}$ for certain $\alpha_i \in \N$; cf. the proof of Theorem~\ref{thm:quotient_rack} for notations. The image $J(Z_{(X,r)})$, and hence the subgroup $Z_{(X,r)}$ itself, is then free abelian of rank~$K_{r}$.

Let us check that Lemma~\ref{lem:cocycles} applies to the bijective group $1$-cocycle $J \colon G_{(X,r)} \to G'_{(X,\op_r)}$, and to the relations
\[\cR_{G_{(X,r)}}=\{\, x^{[d_x]} = 1 \, | \, x \in X \,\}, \qquad\qquad \cR_{G'_{(X,\op_r)}}=\{\, x^{d_x} = 1  \, | \,  x \in X\,\}.\] 
The first condition follows from~\eqref{eqn:PowersActTrivially}. The second one reads
\[\{\, x^{d_x} \ract z  \, | \,  x \in X\,\} \, = \, \{\, x^{d_x}  \, | \,  x \in X\,\}\]
for all $z \in X$. Lemmas \ref{lem:T_r} and \ref{lem:degrees_coinside} yield $x^{[d_x]} \ract  z = \tau_z(x)^{[d_{x}]} = \tau_z(x)^{[d_{\tau_z(x)}]}$. Since the $\tau_z$ are bijective, we are done.

So, Lemma~\ref{lem:cocycles} yields a bijective group $1$-cocycle 
\[\oJ \colon \oG_{(X,r)} \to \raisebox{.5mm}{$G'_{(X,\op_r)}$}\big/\raisebox{-.5mm}{$\{x^{d_x} = 1\} $}.\]
The group $\oG'_{(X,\op_r)}$ from Remark~\ref{rem:quotient_rack_sym} is in general only a quotient of this second group, which we denote by $H$. But the argument from the proof of Theorem~\ref{thm:quotient_rack}  repeats verbatim for $H$, and yields its finiteness and injectivity preservation, which imply finiteness and injectivity preservation for $\oG_{(X,r)}$. 
\end{proof}

\begin{exa}
Let $(X,\op)$ be a rack. Consider the solution $(X,r'_{\op})$. One has $x^y=x$ for all $x,y \in X$, hence $J=\Id$, and $y^{[m]} = y^m \in G_{(X,r'_{\op})}$ for all $y \in X$, $m \in \N$. Further, $x \racts y = x$ and $y \lacts x = x \op y = \rho_y(x)$. Then $d_y$ is $2$ if $\rho_y=\Id_X$, and the degree of $\rho_y$ otherwise. But this is precisely $D_y$. Thus the quotient $\oG_{(X,r'_{\op})}$ from Theorem~\ref{thm:quotient} recovers the quotient $\oG'_{(X,\op)}$ from Remark~\ref{rem:quotient_rack_sym}.
\end{exa}

\begin{exa}
For an involutive solution $(X,r)$, the structure rack is trivial, i.e., $\rho_y=\Id_X$ for all $y \in X$. Thus $d_y$ is the smallest even~$d$ satisfying $x \racts y^{[d]} = y^{[d]} \lacts x = x$ for all $x$. The quotient $\oG_{(X,r)}$ is in this case very close to that from~\cite{DehCycleSet}. There the same power~$d$ was taken for all $y$, and the only condition imposed was $x \racts y^{[d]} = x$ for all $x,y$. Note that the symmetric condition $y^{[d]} \lacts x = x$ is a consequence thereof, because of the relation between the actions $\racts$ and $\lacts$ in the involutive case. Moreover, if the solution is indecomposable (i.e., has a unique orbit), then Lemma~\ref{lem:degrees_coinside} implies $d_x=d_y$ for all $x,y \in X$. Thus, compared to the quotient from~\cite{DehCycleSet}, our $\oG_{(X,r)}$ looses precision due to the evenness requirement for~$d$, but gains precision for decomposable solutions.
\end{exa}

\begin{rem}\label{r:BrOnStrGrp}
The solution~$r$ on~$X$ induces an infinite invertible non-degenerate YBE solution $R$ on $G_{(X,r)}$. Even better: $G_{(X,r)}$ becomes a \emph{braided commutative group}. The condition $x \racts y^{[d_y]} = y^{[d_y]} \lacts x = x$ is precisely what is needed for $R$ to survive in the quotient $\oG_{(X,r)}$. One thus obtains a rich source of finite braided commutative groups, finite \emph{skew braces} etc.; cf. Remark~\ref{r:braces}. Also, Theorem~\ref{thm:quotient} implies that a finite injective solution $(X,r)$ injects into a finite skew brace, in such a way that $r$ is the restriction to~$X$ of the solution constructed on the skew brace in~\cite{GV}. For involutive solutions (which are always injective), this yields a result of \cite[Remark 7]{CGIS}. 
\end{rem}

\begin{rem}\label{r:QuotientForMonoids}
A weaker version of Theorem~\ref{thm:quotient} holds for structure monoids, with analogous proof. Namely, one gets a normal abelian sub-monoid of finite index. For SD solutions, it is a central sub-monoid. It is not free abelian in general. Injectivity is an automatic property for structure monoids, since they are graded and have degree $2$ relations only.
\end{rem}

We finish this section with an alternative definition of the degrees $d_x$. For this, the formula $\rho_y(x) = x \op_r y  = \tau_y \htau_y^{-1}(x)$ needs to be generalized. For $\overline{y}=y_1\ldots y_n \in X^{\times n}$, put $\rho_{\overline{y}}=\rho_{y_n}\cdots \rho_{y_1}$. Thus, conditions \eqref{it:degreeSD}--\eqref{it:degreeActions} from Definition~\ref{def:degree} read
\begin{align*}
&\rho_{y^d} = \tau_{y^{[d]}}|_X = \sigma_{y^{[d]}}|_X = \Id_X.
\end{align*}

\begin{lem}\label{lem:tau_rho}
For any $\overline{y}=y_1\ldots y_n \in X^{\times n}$, one has 
\[\rho_{\varphi(\overline{y})} =  \tau_{\overline{y}} \htau_{\overline{y}}^{-1}, \qquad\qquad \text{ where }\: \varphi(y_1\ldots y_n)=y_1^{y_2\ldots y_n}\, \ldots\, y_{n-1}^{y_n}\,\, y_n.\]
\end{lem}

\begin{proof}
We use induction on~$n$. The case $n=1$ follows from the definition of~$\rho$. To get from $n-1$ to $n$, we need Lemma~\ref{lem:TauActsByRackMor}, rewritten as 
\[\tau_z \rho_y = \rho_{\tau_z(y)} \tau_z\]
 for all $y,z \in X$. Then, for any $\overline{y}=y_1\ldots y_n \in X^{\times n}$, one has 
 \begin{align*}
\tau_{\overline{y}} \htau_{\overline{y}}^{-1} &= \tau_{y_n}\cdots \tau_{y_1} \htau_{y_1}^{-1}\cdots \htau_{y_n}^{-1} = \tau_{y_n}\cdots \tau_{y_2} \rho_{y_1} \htau_{y_{2}}^{-1}\cdots \htau_{y_n}^{-1}\\
& = \rho_{\tau_{y_n}\cdots \tau_{y_2}(y_1)} \, \tau_{y_n}\cdots \tau_{y_2} \htau_{y_2}^{-1}\cdots \htau_{y_n}^{-1} = \rho_{y_1^{y_2\ldots y_n}}\, \rho_{\varphi(y_2\ldots y_n)} = \rho_{\varphi(\overline{y})}. 
 \end{align*}
We used the induction hypothesis for $y_2\ldots y_n$.
\end{proof}

An immediate corollary is the following observation.
\begin{pro}\label{pro:DegreeEquivalentDef}
Any two of the relations $\rho_{y^d}|_X = \Id_X$, $\tau_{y^{[d]}}|_X = \Id_X$, $\htau_{y^{[d]}}|_X = \Id_X$, each considered for  all $y \in X$, imply the third one.
\end{pro}

\section{Applications}\label{s:Applications}

Theorem~\ref{thm:quotient} has several immediate implications.
\begin{cor}\label{cor:SmallerQuotient}
Let $(X,r)$ be a solution. The subgroup
\[Z^0_{(X,r)}= \,\{\, g \in G_{(X,r)} \,|\, \forall x\in X,\, x \racts g = g \lacts x = x \,\}\]
of $G_{(X,r)}$ is normal, of finite index, and abelian of rank~$K_{r}$.
\end{cor}

\begin{proof}
One can see $Z^0_{(X,r)}$ as the intersection of the kernels of the right and the left $G_{(X,r)}$-actions on~$X$. Hence it is a normal subgroup of $G_{(X,r)}$. Further, $r$ induces a YBE solution $R$ on $G_{(X,r)}$ compatible with its group structure (Remark~\ref{r:BrOnStrGrp}), which yields a right and a left $G_{(X,r)}$-actions on itself, whose kernels include the kernels of the $G_{(X,r)}$-actions on~$X$. For any $g,h \in G_{(X,r)}$, one deduces $R(g,h)=(h,g)$, hence $gh=hg$. So, the subgroup $Z^0_{(X,r)}$ is abelian. Finally, it contains $Z_{(X,r)}$ as a subgroup, hence its finite index and its rank~$K_{r}$.
\end{proof}

We thus recover Theorem 2.6 from \cite{Sol} and Proposition 6 from \cite{LYZ}.

\begin{cor}\label{cor:VirtAb}
Let $(X,r)$ be a solution. Its structure group $G_{(X,r)}$ is
\begin{enumerate}
\item virtually $\Z^{K_{r}}$, where $K_{r}$ is the number of orbits of the structure rack $(X,\op_r)$;
\item linear;
\item residually finite. 
\end{enumerate} 
\end{cor}

\begin{proof}
Point 1 is a reformulation of Theorem~\ref{thm:quotient}. Further, consider the obvious degree $K_{r}$ faithful representation of $Z_{(X,r)} \cong \Z^{K_{r}}$ over $\R$. Its induced representation is faithful of degree $K_{r}|\oG_{(X,r)}|$. Finally, all finitely generated linear groups are residually finite (Mal'cev 1940).
\end{proof}

Another property of structure groups follows from our explicit description of the subgroups $Z_{(X,r)}$:

\begin{pro}\label{pro:abelianization}
Let $(X,r)$ be a solution. The abelianization $\Ab G_{(X,r)}$ of its structure group is of rank $k_r$, which is the number of its orbits (cf. Definition~\ref{def:OrbitsSolution}).
\end{pro}

\begin{proof}
Put $k=k_r$, $K=K_r$. Let $\O_1, \ldots, \O_k$ be the orbits of $(X,r)$, and let $\RO_1, \ldots, \RO_K$ be the orbits of $(X,\op_r)$. Recall that the latter refine the former; this yields a surjection $m \colon \{1,\ldots,K\} \twoheadrightarrow \{1,\ldots,k\}$. The free abelian groups $\oplus_i \Z \O_i \cong \Z^k$ and $\oplus_i \Z \RO_i \cong \Z^K$ will be abusively denoted by $\Z^k$ and $\Z^K$ respectively. Thus, the map~$m$ induces a group morphism $\mu \colon \Z^K \twoheadrightarrow \Z^k$. For the reader's convenience, we summarize all group morphisms used in this proof in a commutative diagram:
\[\xymatrix@!0 @R=1.2cm @C=2.5cm{
\Z^K \ar@{^{(}->}[r]^{\kappa} \ar[rd]^{\kappa'} \ar@{->>}[d]_-{\mu} & G \ar@{->>}[rd]^{\nu}  \ar@{->>}[d]^-{\pi} &\\
\Z^k \ar[r]_-{\kappa''} & \Ab G \ar@{->>}[r]_-{\nu'} & \Z^k}\]
Choose representatives $y_i$ of the orbits~$\RO_i$. Put $d_i = d_{y_i}$. In the proofs of Theorems~\ref{thm:quotient_rack} and~\ref{thm:quotient}, we showed that the group morphism $\kappa\colon \Z^K \to G_{(X,r)}$ defined by $\RO_i \mapsto y_i^{[d_i]}$ is injective. Consider also the natural projection $\pi \colon G_{(X,r)}\twoheadrightarrow \Ab G_{(X,r)}$, and the composition $\kappa'=\pi\kappa$. Further, the morphism $\nu \colon G_{(X,r)}\twoheadrightarrow \Z^k$ sending every $x$ to its orbit is well defined and induces $\nu' \colon \Ab G_{(X,r)}\twoheadrightarrow \Z^k$. Lemmas~\ref{lem:T_r} and~\ref{lem:degrees_coinside} imply the relations $y^{[d_y]}z = z \tau_z(y)^{[d_{\tau_z(y)}]}$ and $zy^{[d_y]} = \htau_{z}^{-1}(y)^{[d_{\htau_{z}^{-1}(y)}]}z$ in $G_{(X,r)}$, hence $y^{[d_y]} = \tau_z(y)^{[d_{\tau_z(y)}]}=\htau_{z}^{-1}(y)^{[d_{\htau_{z}^{-1}(y)}]}$ in $\Ab G_{(X,r)}$. Therefore, $\kappa'$ induces a group morphism $\kappa'' \colon \Z^k \to \Ab G_{(X,r)}$. Now, the composition $\nu'\kappa''$ multiplies each $\O_i$ by $d_j$, for any $j \in m^{-1}(i)$. Hence $\kappa''$ is injective. To conclude, it remains to show that $\kappa''(\Z^k)$ is of finite index in $\Ab G_{(X,r)}$. This follows from the finiteness of $\oG_{(X,r)}$ and from the surjection 
\[\oG = G/\kappa(Z^K) \twoheadrightarrow \Ab G / \pi\kappa(Z^K) = \Ab G / \kappa''\mu(Z^K) = \Ab G / \kappa''(\Z^k).\qedhere\]
\end{proof}

It would also be interesting to understand the torsion of $\Ab G_{(X,r)}$. It is trivial for SD solutions. However, in general this is not the case:

\begin{exa}\label{exa:1}
Consider the involution $\psi \colon a \leftrightarrow b$ on the set $X = \{a,b\}$. The map $r(x,y) = (\psi(y),\psi(x))$ yields an involutive solution. Its structure group is 
\[G_{(X,r)} = \langle\, a,b \,|\, a^2=b^2 \,\rangle,\hspace*{2cm} \Ab G_{(X,r)} \cong \Z \times \Z_2.\] 
Let us also describe the finite quotient $\oG_{(X,r)}$. The degrees are $d_a=d_b=2$. The corresponding $T$-twisted powers are $J^{-1}(a,a) = (b,a)$, $J^{-1}(b,b) = (a,b)$. One gets \[\oG_{(X,r)} = \langle\, a,b \,|\, a^2=b^2, ab=ba = 1 \,\rangle \cong \Z_4,\]
where $a$ is sent to $1$ and $b$ to $-1$. Thus our solution is injective. Any attempt to extend $Z_{(X,r)}$ would destroy the injectivity of $X \to \oG_{(X,r)}$ here.
\end{exa}

\begin{exa}\label{exa:2}
A slight modification of the previous example produces torsion-free $\Ab G_{(X,r)}$. Namely, extend $\psi$ to $Y = \{a,b,c\}$ by $\psi(c)=c$. As before, $r(x,y) = (\psi(y),\psi(x))$ is an involutive solution, with 
\[G_{(X,r)} = \langle\, a,b,c \,|\, a^2=b^2, ac=cb, bc=ca \,\rangle \cong \langle\, a,c \,|\, ac^2 = c^2a, a^2c=ca^2 \,\rangle,\] 
\[\Ab G_{(X,r)} \cong \Z^2.\] 
All degrees are $2$, and $J^{-1}(c,c) = (c,c)$. So, \[\oG_{(X,r)} = \langle\, a,b \,|\, a^2=b^2, ac=cb, bc=ca, ab=ba =c^2=1 \,\rangle \cong \Z_4 \rtimes \Z_2,\]
where $\Z_2$ acts on $\Z_4$ by the sign change. This equivalence works as follows: $a \mapsto (1,0)$, $b \mapsto (-1,0)$, $c \mapsto (0,1)$. Here again $Z_{(X,r)}$ is a maximal subgroup satisfying all requirements of Theorem~\ref{thm:quotient}.
\end{exa}

Recall that according to Proposition~\ref{pro:birack_biquandle}, the structure group does not change if a solution is replaced with a biquandle quotient. We will now construct an even smaller quotient which is injective and still has the same structure group. As a result, in a study of structure groups one can work with injective biquandles only. While the natural definition of this quotient involves comparing elements of the (infinite) group $G_{(X,r)}$, an alternative definition compares elements of our finite quotients $\oG_{(X,r)}$ only, which is realizable by a computer.

\begin{pro}\label{pro:birack_injective}
For a solution $(X,r)$, define an equivalence relation by $x \approx x'$ if and only if $\iota(x)=\iota(x')$ in $G_{(X,r)}$. Then $r$ induces an injective solution $r'$ on $X/\!\approx$. Moreover, the quotient map $X \twoheadrightarrow X/\!\approx$ induces a group isomorphism
\[G_{(X,r)} \overset{\sim}{\longrightarrow} G_{(X/\!\approx,\, r')}.\]
\end{pro}

\begin{proof}
For $x,x'\in X$, the relation $x \approx x'$ implies $\sigma_x=\sigma_{x'}$ and $\tau_x=\tau_{x'}$. So in $G_{(X,r)}$ one has 
\[\sigma_x(y)\tau_y(x) = xy = x'y = \sigma_x(y)\tau_y(x'),\]
hence $\tau_y(x) \approx \tau_y(x')$. Similarly, $\sigma_y(x) \approx \sigma_y(x')$. As a result, the induced map $r'$ is well-defined. The remaining assertions are straightforward.
\end{proof}

\begin{defn}
The biquandle $(X/\!\approx,\, r')$ from the proposition will be called the \emph{induced injective solution (IIS)} of $(X,r)$, denoted by $\IIS (X,r)$.
\end{defn}

By Remark~\ref{rem:IndBiquUniv} and Lemma~\ref{lem:Inj_Biqu}, the IIS of $(X,r)$ is a quotient of its induced biquandle. Also, the IIS construction enjoys a universal property analogous to that from Remark~\ref{rem:IndBiquUniv}. Moreover, it defines a functor $\SolCat \to \InjSol$ to the category of injective solutions, which is a retraction for the inclusion functor. This construction restricts to racks and yields a functor $\Rack \to \InjRack$ to the category of injective racks (which are necessarily quandles). By the bijectivity of the guitar map~$J$, the structure rack construction intertwines these ``injectivization'' functors. One gets a commutative diagrams of functors analogous to those from Propositions~\ref{pro:commdiag_Sol_Q_BQ} and~\ref{pro:commdiag_R_Q_BQ}: 
\[\xymatrix@!0 @R=1cm @C=1.5cm{
\Rack \ar[rr] \ar[d] \ar@/_4pc/_-{\Id}[dd] && \InjRack \ar[d] \ar@/^4pc/^-{\Id}[dd]\\
\SolCat \ar[rr] \ar[d] && \InjSol \ar[d]\\
\Rack \ar[rr]&& \InjRack}\]

The injectivity preservation argument from the proof of Theorem~\ref{thm:quotient_rack} has the following by-product:

\begin{lem}\label{lem:TestIISol}
Elements $x,x' \in X$ yield the same element of $G_{(X,r)}$ if and only if they yield the same element of $\oG_{(X,r)}$.
\end{lem}

As a result, the relation $x \approx x'$ can be tested in the finite group $\oG_{(X,r)}$.

\section{Orderability}\label{s:Orderability}

Recall that a group $G$ is \emph{left-orderable} if it can be endowed with a total order stable by left translations: $a < b \,\Longrightarrow\, ca < cb$. If the order can be chosen to be also stable by right translations, then $G$ is called \emph{bi-orderable}. Orderability has useful algebraic implications, the simplest of which is torsion-freeness. It has also remarkable connexions with topology, geometry, dynamics, probability \cite{ShortWiest,OrderingBraids,GroupsOrdersDynamics}. This section explores orderability properties of structure groups.

We need a general result on bi-orderability:
\begin{pro}\label{p:VirtAbBiord}
Let $G$ be a virtually abelian finitely generated group. Then $G$ is bi-orderable if and only if it is free abelian.	
\end{pro}

\begin{proof}
The ``if'' implication is straightforward. Conversely, assume that $G$ is bi-orderable but non-abelian. Take $g,h\in G$ such that $gh>hg$, i.e. $h^{-1}gh > g$. This implies $h^{-1}g^nh > g^n$, hence $g^nh > hg^n$ for all $n \in \N$. Similarly, one proves $g^nh^n>h^ng^n$. Since $G$ has an abelian subgroup $H$ of finite index, there exists an $n \in \N$ such that $g^n,h^n \in H$, hence $g^nh^n=h^ng^n$, contradiction.
\end{proof}

Coupled with Theorem~\ref{thm:quotient}, this result characterizes bi-orderable structure groups:
\begin{thm}\label{thm:StrGrBiord}
Let $(X,r)$ be a solution. Then $G_{(X,r)}$ is bi-orderable if and only if it is free abelian. In this case,
\begin{enumerate}
 \item the orbits of the solution $(X,r)$ and of its structure rack $(X,\op_r)$ coincide, and $k_r=K_r$;
 \item the group morphism $\nu \colon G_{(X,r)}\twoheadrightarrow \Z^{k_r}$ sending every $x$ from the orbit $\O_i$ to the generator $e_i$ of $\Z^{k_r}$ is an isomorphism.
\end{enumerate}
\end{thm}

\begin{proof}
It remains to prove the last two statements. We use notations from the proof of Proposition~\ref{pro:abelianization}.
\begin{enumerate}
\item We will prove a slightly more general statement. If $G_{(X,r)}$ is abelian, then the quotient map  $\pi \colon G_{(X,r)}\twoheadrightarrow \Ab G_{(X,r)}$ is the identity. Since the rank of $\Ab G_{(X,r)}$ is~$k_r$, this yields $G_{(X,r)} = \Ab G_{(X,r)} \cong \Z^{k_r} \times A$ for some finite abelian group~$A$. Recall the map $\kappa \colon \Z^{K_r} \hookrightarrow G_{(X,r)} \cong \Z^{k_r} \times A$. Its injectivity implies $K_r \leq k_r$. Since the $K_r$ orbits of $(X,\op_r)$ refine the $k_r$ orbits of $(X,r)$, this inequality means that these two partitions into orbits coincide.
\item Now, suppose $G_{(X,r)}$ free abelian. Then $\nu$ surjects $G_{(X,r)} \cong \Z^{k_r}$ onto $\Z^{k_r}$, and is thus an isomorphism. \qedhere
\end{enumerate}
\end{proof}

Structure groups of \emph{trivial} solutions $r(x,y)=(y,x)$ are free abelian. In general the converse is false: for instance, the structure group of the rack $(\Z, \, x \op y = x+1)$ is $\Z$. However, it holds for involutive solutions:

\begin{pro}\label{p:InvolStrGrBiord}
Let $(X,r)$ be an involutive solution. Then $G_{(X,r)}$ is abelian if and only if $r$ is trivial. 
\end{pro}

\begin{proof}
Assume that $G_{(X,r)}$ is abelian. Our proof of Theorem~\ref{thm:StrGrBiord} identifies the two partitions into orbits. Recall that the induced rack of an involutive solution is trivial, $x \op_r y = x$, so all its orbits are one-element. Hence all orbits of $(X,r)$ are one-element as well. This means $\sigma_x=\tau_x=\Id$ for all $x \in X$, hence $r$ is trivial.
\end{proof}

Thus, among involutive solutions, only the trivial ones have bi-orderable structure groups.

To deal with left-orderability, we need a general notion of MP solutions, which is well studied in the literature for the particular case of involutive solutions.

\begin{lem}\label{lem:Retraction}
Let $(X,r)$ be a solution. Define an equivalence relation on~$X$ by 
\[x\sim x' \quad \Longleftrightarrow \quad \sigma_x=\sigma_{x'} \;\&\; \tau_x=\tau_{x'}.\] 
Then $r$ induces a solution~$\overline{r}$ on $\overline{X} = X/{\sim}$. 
\end{lem}

\begin{proof}
We only have to check that $\overline{r}$ is well defined. This boils down to verifying $\sigma_{\phi(x)}=\sigma_{\phi(x')}$, $\tau_{\phi(x)}=\tau_{\phi(x')}$ for $x \sim x'$ and $\phi \in \{\sigma_y,\tau_y\}$, $y \in X$. But this directly follows from Equations \eqref{eqn:tau_action}--\eqref{eqn:sigma_action}.
\end{proof}

\begin{defn}\label{def:retract}
The solution from Lemma~\ref{lem:Retraction} is called the \emph{retraction} of $(X,r)$, denoted by $\Ret(X,r)=(\overline{X},\overline{r})$. A solution is called \emph{multipermutation (MP) of level $n$} if $n$ is the minimal non-negative integer such that $\#\Ret^n(X,r)=1$. A solution is called \emph{retractable} if $\overline{X}$ is smaller than $X$, and \emph{irretractable} otherwise.
\end{defn}

\begin{exa}
The trivial one-element solution is the only MP solution of level $0$. Being of level $1$ means having the form $r(x,y)=(f(y),g(x))$, where $f$ and $g$ are commuting symmetries of~$X$.
\end{exa}

\begin{exa}
For involutive solutions, relation $\sigma_x=\sigma_{x'}$ implies $\tau_x=\tau_{x'}$. We thus recover the definition of~\cite{ESS}. 
\end{exa}

\begin{exa}
For an SD solution $(X,r_{\op})$, $x \sim x'$ means $z \op x = z \op {x'}$ for all $z \in X$. Then $\op$ induces a rack operation $\overline{\op}$ on $\overline{X}=X/{\sim}$, and the solution $\overline{r_{\op}}$ is associated to this rack operation: $\overline{r_{\op}} = {r_{\overline{\op}}}$. The pair $(\overline{X},\overline{\op})$ is called the \emph{retraction} of $(X,\op)$. A rack is called \emph{MP of level $n$} if the associated solution is so.
\end{exa}

\begin{lem}\label{lem:quotient}
The retraction $X/{\sim}$ of a solution is a quotient of its induced injective solution $X/{\approx}$. 
\end{lem}

\begin{proof}
Since the collections of maps $\sigma$ and $\tau$ define $G_{(X,r)}$-actions on $X$, the relation $x \approx x'$ implies $x \sim x'$.
\end{proof}

Multipermutation involutive solutions are known to be the only involutive
solutions with left-orderable structure group. One of the implications of the
following theorem was proved in~\cite{ChouOrd}; see also~\cite[Proposition
4.2]{MR2189580}. The other was proved in~\cite[Theorem 2.1]{BCV}.

\begin{thm}
	\label{thm:LO_Invol}
Let $(X,r)$ be an involutive solution. Then $G_{(X,r)}$ is left-orderable if and only if $(X,r)$ is multipermutation.
\end{thm}

Moreover, the space of left orders of the structure group of a non-trivial MP involutive solution with $\# X \geq 3$ is known to be extremely rich \cite{ChouOrd}.

\begin{exa} Examples \ref{exa:1} and \ref{exa:2} describe structure groups of two level $1$ involutive solutions. These groups are then left-orderable, and have the following remarkable feature: the element $a^2=b^2$ has at least two square roots.
\end{exa}

A group $G$ is said to be \emph{diffuse} if for every finite non-empty subset $A$ of~$G$ there exists an $a\in A$ such that for all $g\in G\setminus\{1\}$, either $ga\not\in A$ or $g^{-1}a\not\in A$. Such a group satisfies the \emph{unique product property (UPP)}: for every finite non-empty subsets $A,B$ of~$G$ there is an element $x$ which can be written uniquely as $x=ab$ with $a\in A$, $b\in B$. These technical properties imply Kaplansky's conjectures for~$G$, and were conceived as efficient tools for proving them. See \cite{Diffuse} for an up-to-date survey of these questions.

We now show that for structure groups of involutive solutions, being diffuse and left-orderable is the same. This answers a question of Chouraqui \cite{ChouOrd}. 

\begin{thm}\label{thm:diffuse}
Let $(X,r)$ be an involutive solution. Then $G_{(X,r)}$ is diffuse if and only if $(X,r)$ is MP.
\end{thm}

\begin{proof}
Implication ``left-orderable $\,\Longrightarrow \,$ diffuse'' is classical. Indeed, for a finite non-empty subset $A$ of~$G$, the minimal element $a$ of~$A$ satisfies the property from the definition. Further, our group $G_{(X,r)}$ is virtually abelian. Since every virtually abelian group is amenable, by \cite[Theorem 6.4]{LWM}, such a group is diffuse if and only if it is locally indicable (that is, any of its finitely generated non-trivial subgroups surjects onto~$\Z$); see also \cite[Theorem 3.3]{Diffuse}. To conclude, use the classical implication ``locally indicable $\,\Longrightarrow \,$ left-orderable'' (see \cite{GroupsOrdersDynamics}, or any other textbook on orderable groups), and Theorem~\ref{thm:LO_Invol}.
\end{proof}

\begin{question}
Can structure groups of irretractable involutive solutions satisfy the UPP?
\end{question}
A positive answer would yield an example of a non-diffuse group with UPP, solving an open question. 

\begin{exa}
One of the simplest irretractable involutive solutions is given by $X=\{1,2,3,4\}$ and $r(x,y)=(\sigma_x(y),\sigma^{-1}_{\sigma_x(y)}(x))$, where
	\[
		\sigma_1=(23),\quad
		\sigma_2=(14),\quad
		\sigma_3=(1243),\quad
		\sigma_4=(1342).
	\]
In \cite[Example 8.2.14]{MR2301033} Jespers and Okni{\'n}ski proved that the UPP is false for its structure group. 
\end{exa}

Let us now turn to our second favourite class of solutions. For SD solutions, left-orderability turns out to be equivalent to bi-orderability:

\begin{thm}\label{thm:LO_SD}
For a finite rack $(X,\op)$, the following statements are equivalent:
	\begin{enumerate}
		\item $G_{(X,\op)}$ is bi-orderable; 
		\item $G_{(X,\op)}$ is left-orderable; 
		\item $G_{(X,\op)}$ has no torsion;
		\item $G_{(X,\op)}$ is free abelian;
		\item $G_{(X,\op)}$ is abelian;
		\item the induced injective rack $X/\!\approx$ of $(X,\op)$ is trivial.
	\end{enumerate}
Any of these statements implies that $(X,\op)$ is a multipermutation rack of level at most~$2$.	
\end{thm}

From this theorem we learn that structure groups of racks yield no new examples of left-orderable groups. On the bright side, for this class of groups we obtain an interesting dichotomy: either they are free abelian, or they are non-abelian and have torsion. This dichotomy is decided by testing the relation $x \approx x \op y$ for all $x,y \in X$, which, according to Lemma~\ref{lem:TestIISol}, can be tested in the finite group $\oG_{(X,\op)}$.

\begin{proof}
Implications $4 \Rightarrow 1 \Rightarrow 2 \Rightarrow 3$ are classical. Let us show $3 \Rightarrow 4$. Assume that $G_{(X,\op)}$ has no torsion. Operation $\op$ induces a $G_{(X,\op)}$-action on~$X$, whose kernel is contained in the center of $G_{(X,\op)}$. So this center is of finite index. According to Schur's theorem, the commutator subgroup of $G_{(X,\op)}$ is then finite (cf. \cite[Theorem 5.32]{Rotman}). Since $G_{(X,\op)}$ has no torsion, this subgroup is trivial, and $G_{(X,\op)}$ is free abelian. 

Now, assume $G_{(X,\op)}$ abelian. For any $x,y\in X$, from $yx=xy=y (x \op y)$ in $G_{(X,\op)}$ one deduces $x \approx x \op y$. Hence the induced injective rack $(X/\!\approx,\op')$ of $(X,\op)$ is trivial. Further, since the structure groups of $(X,\op)$ and $(X/\!\approx,\op')$ are isomorphic, the triviality of $(X/\!\approx,\op')$ means that $G_{(X/\!\approx,\op')}$, and hence $G_{(X,\op)}$, is free abelian. We get $5 \Rightarrow 6 \Rightarrow 4$. Implication $4 \Rightarrow 5$ is trivial.

Finally, by Lemma~\ref{lem:quotient}, if $(X/\!\approx,\op')$ is trivial then so is $\Ret(X,\op)$. Hence $\Ret^2(X,\op)$ has one element only.
\end{proof}

Note that some level $2$ MP racks do not satisfy the conditions of the theorem:

\begin{exa}
The set $X = \{a,b,c,d\}$ with right translations $\rho_a=\rho_b\colon c \leftrightarrow d$ and $\rho_c=\rho_d\colon a \leftrightarrow b$ is a level $2$ quandle. Consider the group 
\[H = \langle\, a,c,t \,|\, a^2=c^2=t^2=1, \, t \text{ central},\, ac = tca \,\rangle.\]
It is isomorphic to $(\Z_2 \times \Z_2) \rtimes \Z_2$, where $a,t,c$ generate the three copies of $\Z_2$, in this order, and $c$ acts on $\Z_2 \times \Z_2$ by sending $a^{\alpha}t^{\tau}$ to $a^{\alpha}t^{\tau+\alpha}$. In particular, $a,at,c,ct$ are four distinct elements of~$H$. Now, a surjection $G_{(X,\op)} \twoheadrightarrow H$ can be defined as follows: $a \mapsto a, b \mapsto at, c \mapsto c, d \mapsto ct$. (In fact, $H$ is the finite quotient $\oG_{(X,\op)}$. That is how it appeared in our argument!) Therefore, $a,b,c,d$ are four distinct elements of $G_{(X,\op)}$, and our quandle is injective. In particular, $X/\!\approx $ is just $X$, which is non-trivial. Also, $ab^{-1}$ is a $2$-torsion element in $G_{(X,\op)}$. In this example, the finite quotient $\oG_{(X,\op)}$ was essential to determine that our structure group does not satisfy the properties from the theorem.
\end{exa}

Among injective quandles, only trivial ones satisfy the conditions of the theorem. This is not true for non-injective quandles:

\begin{exa}
Consider the quotient of the quandle from the previous example by the relation $b=c$. This is a $3$-element quandle with the structure group~$\Z^2$.
\end{exa}

Left-orderability question for general structure groups remains open:

\begin{question}
Let $(X,r)$ be a non-involutive injective solution. Can its structure group $G_{(X,r)}$ be left orderable? 
\end{question} 

We suspect the answer to be negative. Indeed, the structure rack $(X,\op_r)$ of such a solution is injective and non-trivial, so, by Theorem~\ref{thm:LO_SD}, the group $G_{(X,\op_r)}$ has torsion. The groups $G_{(X,r)}$ and $G_{(X,\op_r)}$ being related by the bijective $1$-cocycle~$J$, this must have serious implications for $G_{(X,r)}$.

\appendix
\section{Size $3$ biquandles}\label{s:Size3}

To illustrate our results, we will classify all solutions whose structure racks are size $3$ quandles. We will describe their structure groups, and finite quotients $\overline{G}$ thereof. We will always work up to solution/quandle isomorphism. While reading this appendix, the reader might keep in mind the following question.

\begin{question}
Given a finite rack $(X,\op)$, how can one construct all solutions having $(X,\op)$ as their right structure rack?
\end{question} 

There are precisely three quandle structures on the set $X=\{0,1,2\}$:
\begin{enumerate}
\item the trivial quandle $T$: $\rho_x = \Id$, $K_{\op}=3$, $G_T \cong \Z^3$, $\overline{G}_T \cong \Z_2^3$, $\Iso_T=1$, the quandle is injective and MP of level~$1$;
\item the two-orbit quandle $S$: $\rho_0=(12)$, $\rho_1 = \rho_2 = \Id$, $K_{\op}=2$, $G_S \cong \Z^2$, $\overline{G}_S \cong \Z_2^2$, $\Iso_S=3$, the quandle is not injective and is MP of level~$2$;
\item the dihedral quandle $D$: $\rho_0=(12)$, $\rho_1=(02)$, $\rho_2=(01)$, $K_{\op}=1$, $\Iso_D=1$, the quandle is injective (as we shall now see) and irretractable.
\end{enumerate}
This description uses the maps $\rho_y \colon x \mapsto x \op y$, the notation $K_{\op}=\#\Orb(X,\op)$ for the number of orbits of a quandle, and the notation $\Iso_Q$ for the number of quandle structures on~$X$ isomorphic to $Q$. 

The structure group of $D$ is a quotient of the braid group $B_3$:
\[G_D \cong \langle \, 1,2\, | \, 121=212, 1^2=2^2\, \rangle \cong \raisebox{.1cm}{$B_3$} \big/ \raisebox{-.1cm}{$1^2=2^2$}, \hspace*{1cm} \Ab G_D \cong \Z.\]
All $x \in D$ are of degree $D_x=2$, so the finite quotient from Theorem~\ref{thm:quotient_rack} is the symmetric group: $\overline{G}_D \cong S_3$. The injectivity of $D$ can be easily tested in~$S_3$. In fact, $D$ is isomorphic to the conjugation class of transpositions in~$S_3$.

The kernel of the surjection $\pi \colon G_D \twoheadrightarrow S_3$ is freely generated by the central element $1^2=2^2$; freeness follows from the surjection $p \colon G_D \twoheadrightarrow \Z$, \, $1,2 \mapsto 1$. So, $G_D$ is a central extension of~$S_3$:
\[0 \to \Z \to G_D \to S_3 \to 0.\]
Since $G_D$ surjects onto the non-abelian group $S_3$, it is non-abelian. It has torsion: the element $1^{-1}2$ is non-trivial since $\pi(1^{-1}2) \neq \Id$, and $\pi((1^{-1}2)^3)=\Id$ implies $(1^{-1}2)^3 = 1^{2n}$ for some $n \in \Z$, which is $0$ because of the surjection~$p$. We obtain an elementary illustration of the main assertions of Theorem~\ref{thm:LO_SD}. 

Given a finite rack $(X,\op)$, a \emph{chain} in $(X,\op)$ is a sequence $(x_i) \in X^{\Z}$ such that $x_{i-1} \op x_i = x_{i+1}$ for all $i \in \Z$. It is periodic. The \emph{period pattern} of a rack is the multi-set of periods of all its chains. These periods sum up to $|X|^2$. For instance, the period patterns of size $3$ quandles are:
\begin{enumerate}
\item $1,1,1,2,2,2$ for $T$;
\item $1,1,1,2,4$ for $S$;
\item $1,1,1,3,3$ for $D$.
\end{enumerate}

Similarly, a \emph{chain} in a finite solution $(X,r)$ is a sequence of pairs $((x_i,y_i)) \in (X\times X)^{\Z}$ such that $r(x_i,y_i) = (x_{i+1},y_{i+1})$ for all $i \in \Z$. It is periodic. Moreover, $(x_i)$ is a chain in $(X,\lop_r)$, and $(y_i)$ is a chain in $(X,\op_r)$. Since $r$ is non-degenerate, the periods of both chains coincide with that of $((x_i,y_i))$. Hence a period-respecting bijection between the chains of the left and the right structure racks of $(X,r)$. 

These chain bijections allow us to classify size $3$ biquandles. Details are tedious but straightforward, and are omitted here.

Recall that a solution $(X,r)$ is called \emph{decomposable} if $X=Y \sqcup Z$, $Y \neq \emptyset \neq Z$, and $r$ restricts to both $Y \times Y$ and $Z \times Z$.

\begin{enumerate}
\item Structure quandle $T$. That is, one classifies involutive solutions of size $3$. They are always injective.
\begin{enumerate}
\item The trivial solution $r(x,y)=(y,x)$. 
\[G \cong \Ab G \cong \Z^3, \hspace*{1cm}\overline{G}\cong \Z_2^3.\]
It is MP of level~$1$, with the number of orbits $k_r=3$.
\item The unique indecomposable solution $r(x,y)=(y+1,x-1)$ (cf.~\cite{ESS}). It is MP of level~$1$. Here $k_r=1$, and
\[G \cong \langle \, 0,1 \,|\, 0^3=1^3, (10)^3=1^3 0^3 \,\rangle, \hspace*{.5cm} \Ab G \cong \Z \times \Z_3,\]
We do not see any conceptual description of this structure group or its finite quotient. In particular, the left orderability of~$G$ and the finiteness of~$\overline{G}$ are not obvious from their presentations.
\item Solution from Example~\ref{exa:2}: $r(x,y)=(-y,-x)$. It is MP of level~$1$. Here $k_r=2$,
\[G \cong \langle \, 0,1 \,|\, 01^2=1^2 0, 0^2 1=10^2 \,\rangle, \hspace*{.5cm} \Ab G \cong \Z^2, \hspace*{.5cm} \overline{G}\cong \Z_4 \rtimes \Z_2.\]
\item The two remaining solutions are MP of level~$2$, and have $k_r=2$. The first one is
$r(x,y)=(\sigma_x(y),\sigma_y(x))$, with $\sigma_1=\sigma_2 = (12), \sigma_0=\Id$.
\[G \cong \langle \, 1,2 \,|\, 1^2=2^2 \,\rangle \times \Z, \hspace*{.5cm} \Ab G \cong \Z^2 \times \Z_2,\hspace*{.5cm} \overline{G}\cong \Z_4 \times \Z_2.\]
\item The second one is
$r(x,y)=(\sigma_x(y),\sigma_y(x))$, $\sigma_1=\sigma_2 = \Id, \sigma_0=(12)$.
\[G \cong \Z^2 \rtimes \Z, \hspace*{.5cm} \Ab G \cong \Z^2, \hspace*{.5cm} \overline{G}\cong \Z_2^2 \rtimes \Z_2.\]
In both semidirect products, the action is by component permutation: $1\cdot (a,b)=(b,a)$. 
\end{enumerate}
\item Structure quandle $S$. Here there are four non-isomorphic solutions, but they all yield the same groups:
\[G \cong \Ab G \cong G_{(X,\op)} \cong \Z^2, \, J=\Id, \hspace*{1cm}\overline{G}\cong \Z_2^2.\]
All solutions are decomposable, MP of level~$2$, non-injective, with $k_r=2$. Their induced injective solutions are trivial of size~$2$.
\begin{enumerate}
\item The SD solution $r=r_{\op}$.
\item $r(x,y)=(\sigma_x(y),-x)$, with $\sigma_1=\sigma_2 = (12), \sigma_0=\Id$.
\end{enumerate}
\item Structure quandle $D$. All solutions are indecomposable, irretractable, injective, with $k_r=1$. Their structure groups have $3$-torsion, and are thus not left-orderable.
\begin{enumerate}
\item The SD solution $r=r_{\op}$. Here $G \cong G_D$, $\overline{G} \cong \overline{G}_D$.
\item $r(x,y)=(y+1,1-x-y)$. Here $G \cong G_D$, $\overline{G} \cong G_D/1^6$. The latter group surjects onto $\Z_6$, and is thus different from $\overline{G}_D \cong S_3$.
\item $r(x,y)=(-y,x-y)$. Here $G \cong \Ab G \cong \Z \times \Z_3$, $\overline{G} \cong \Z_6$.
\end{enumerate}
\end{enumerate}
For $S$ and $D$, one should also include the inverses $r^{-1}$ of the listed solutions. Replacing $r$ by $r^{-1}$ does not change the structure group and its finite quotient~$\overline{G}$, but for~$D$ it does change the $1$-cocycle $J$.

\subsection*{Acknowledgments}
The work of L.~V. is partially supported by PICT-2014-1376, MATH-AmSud 17MATH-01, ICTP, ERC advanced grant 320974, and the Alexander von Humboldt Foundation. V.~L. thanks Hamilton Mathematics Institute for support. Both authors are grateful to the Visiting Professors Fund of Trinity College Dublin, which made possible the second author's visit to Dublin. The authors are grateful to the reviewer for constructive suggestions and remarks.
\bibliographystyle{alpha}
\bibliography{refs}
\end{document}